\documentclass[12pt,dvips]{amsart}
  \usepackage{amsfonts, amssymb, latexsym, epsfig}
  \usepackage{euler, amsfonts, amssymb, latexsym, epsfig, epic}
  
  \theoremstyle{plain}
    \newtheorem{theorem}{Theorem}[section]

  \newtheorem*{nonnegativityconjecture}{Nonnegativity conjecture}
  \newtheorem*{upperconjecture}{Upper semicontinuity conjecture}

  \newtheorem{proposition}[theorem]{Proposition}
  
  \newtheorem{lemma}[theorem]{Lemma}
  
\newtheorem{defprop}[theorem]{Definition--Proposition}
  
  \newtheorem{conjecture}[theorem]{Conjecture}
  \newtheorem{problem}[theorem]{Problem}
  
\theoremstyle{definition}
  \newtheorem{definition}[theorem]{Definition}
  \newtheorem{example}[theorem]{Example}
  
  \newtheorem{question}[theorem]{Question}

 \theoremstyle{remark}
  \newtheorem{remark}[theorem]{Remark}
\numberwithin{equation}{section}

\def\Ess{{\mathcal{E}}}

\def\ex{{\rm ex}}
\def\ey{{\rm depth}}
\def\sat{{\rm sat}}

\newcommand{\excise}[1]{}%{$\star$\textsc{#1}$\star$}

\newcommand{\cellsize}{15}
\newlength{\cellsz} 
\newsavebox{\cell}
\sbox{\cell}{\begin{picture}(\cellsize,\cellsize)
\put(0,0){\line(1,0){\cellsize}}
\put(0,0){\line(0,1){\cellsize}}
\put(\cellsize,0){\line(0,1){\cellsize}}
\put(0,\cellsize){\line(1,0){\cellsize}}
\end{picture}}
\newcommand\cellify[1]{\def\thearg{#1}\def\nothing{}%
\ifx\thearg\nothing
\vrule width0pt height\cellsz depth0pt\else
\hbox to 0pt{\usebox{\cell} \hss}\fi%
\vbox to \cellsz{
\vss
\hbox to \cellsz{\hss$#1$\hss}
\vss}}
\newcommand\tableau[1]{\vtop{\let\\\cr
\baselineskip -16000pt \lineskiplimit 16000pt \lineskip 0pt
\ialign{&\cellify{##}\cr#1\crcr}}}
\newcommand{\kellsize}{18}
\newlength{\kellsz} 
\newsavebox{\kell}
\sbox{\kell}{\begin{picture}(\kellsize,\kellsize)
\put(0,0){\line(1,0){\kellsize}}
\put(0,0){\line(0,1){\kellsize}}
\put(\kellsize,0){\line(0,1){\kellsize}}
\put(0,\kellsize){\line(1,0){\kellsize}}
\end{picture}}
\newcommand\kellify[1]{\def\thearg{#1}\def\nothing{}%
\ifx\thearg\nothing
\vrule width0pt height\kellsz depth0pt\else
\hbox to 0pt{\usebox{\kell} \hss}\fi%
\vbox to \kellsz{
\vss
\hbox to \kellsz{\hss$#1$\hss}
\vss}}
\newcommand\ktableau[1]{\vtop{\let\\\cr
\baselineskip -16000pt \lineskiplimit 16000pt \lineskip 0pt
\ialign{&\kellify{##}\cr#1\crcr}}}
\font\co=lcircle10

\def\jr{\smash{\raise2pt\hbox{\co \rlap{\rlap{\char'005} \char'007}}
               \raise6pt\hbox{\rlap{\vrule height5pt}}
               \raise2pt\hbox{\rlap{\hskip4pt \vrule height0.4pt depth0pt
                width5.7pt}}
               \raise2pt\hbox{\rlap{\hskip-9.5pt \vrule height.4pt depth0pt
                width6.2pt}}
               \lower6pt\hbox{\rlap{\vrule height4.5pt}}}}
\def\rj{\smash{\raise2pt\hbox{\co \rlap{\rlap{\char'004} \char'006}}
               \raise6pt\hbox{\rlap{\vrule height5pt}}
               \raise2pt\hbox{\rlap{\hskip4pt \vrule height0.4pt depth0pt
                width5.7pt}}
               \raise2pt\hbox{\rlap{\hskip-9.5pt \vrule height.4pt depth0pt
                width6.2pt}}
               \lower6pt\hbox{\rlap{\vrule height4.5pt}}}}
\def\je{\smash{\raise2pt\hbox{\co \rlap{\rlap{\char'005}
                \phantom{\char'007}}}\raise6pt\hbox{\rlap{\vrule height5pt}}
               \raise2pt\hbox{\rlap{\hskip-9.5pt \vrule height.4pt depth0pt
                width6.2pt}}}}
\def\ej{\smash{\raise2pt\hbox{\co \rlap{\rlap{\char'004}\phantom{\char'006}}}
               \raise2pt\hbox{\rlap{\hskip-9.5pt \vrule height.4pt depth0pt
                width6.2pt}}
               \lower6pt\hbox{\rlap{\vrule height4.5pt}}}}
\def\er{\smash{\raise2pt\hbox{\co \rlap{\rlap{\phantom{\char'005}} \char'007}}
               \raise2pt\hbox{\rlap{\hskip4pt \vrule height0.4pt depth0pt
                width5.7pt}}
               \lower6pt\hbox{\rlap{\vrule height4.5pt}}}}
\def\re{\smash{\raise2pt\hbox{\co \rlap{\rlap{\phantom{\char'004}} \char'006}}
               \raise6pt\hbox{\rlap{\vrule height5pt}}
               \raise2pt\hbox{\rlap{\hskip4pt \vrule height0.4pt depth0pt
                width5.7pt}}}}
\def\+{\smash{\lower6pt\hbox{\rlap{\vrule height17pt}}
                \raise2pt%%%% for ordinary pipe dreams
                \hbox{\rlap{\hskip-9pt \vrule height.4pt depth0pt
                width18.7pt}}}}
\def\hor{\smash{\raise2pt\hbox{\rlap{\hskip-9.5pt \vrule height.4pt depth0pt
                width19.2pt}}}}
\def\ver{\smash{\lower6pt\hbox{\rlap{\vrule height17pt}}}}
\def\ho{\smash{\hbox{\rlap{\vrule height5pt}}
                \raise2pt%%%% for ordinary pipe dreams
                \hbox{\rlap{\hskip-9pt \vrule height.4pt depth0pt
                width18.7pt}}}}

\def\textcross{\ \smash{\lower4pt\hbox{\rlap{\hskip4.15pt\vrule height14pt}}
                \raise2.8pt\hbox{\rlap{\hskip-3pt \vrule height.4pt depth0pt
                width14.7pt}}}\hskip12.7pt}
\def\textelbow{\ \hskip.1pt\smash{\raise2.75pt%
                \hbox{\co \hskip 4.15pt\rlap{\rlap{\char'004} \char'006}
                \lower6.8pt\rlap{\vrule height3.5pt}
                \raise3.6pt\rlap{\vrule height3.5pt}}
                \raise2.8pt\hbox{%
                  \rlap{\hskip-7.15pt \vrule height.4pt depth0pt width3.5pt}%
                  \rlap{\hskip4.05pt \vrule height.4pt depth0pt width3.5pt}}}
                \hskip8.7pt}

\begin{document}
\pagestyle{plain}
\title{Kazhdan-Lusztig polynomials and drift configurations}
\author{Li Li}
\address{Dept. of Mathematics and Statistics\\
Oakland University\\
Rochester, Michigan 48309}
\email{li2345@oakland.edu}

\author{Alexander Yong}
\address{Dept. of Mathematics\\
University of Illinois at Urbana-Champaign\\
Urbana, IL 61801}
\email{ayong@math.uiuc.edu}

\date{August 21, 2010}

\begin{abstract}
The coefficients of the Kazhdan-Lusztig polynomials $P_{v,w}(q)$
are nonnegative integers that are upper semicontinuous on Bruhat order.
Conjecturally, the same properties hold for
$h$-polynomials $H_{v,w}(q)$ of local rings of Schubert varieties. This suggests a parallel between
the two families of polynomials.
We prove our conjectures for Grassmannians, and more generally,
covexillary Schubert varieties in complete flag varieties, by deriving a combinatorial
formula for $H_{v,w}(q)$. We
introduce \emph{drift configurations} to formulate a new and compatible combinatorial
rule for $P_{v,w}(q)$. From our rules
we deduce, for these cases, the coefficient-wise inequality $P_{v,w}(q)\preceq H_{v,w}(q)$.
\end{abstract}

\maketitle
%\tableofcontents
\vspace{-.1in}

\section{Introduction}

\subsection{Overview}
This paper studies two families of polynomials
$\{P_{v,w}(q)\}$ and $\{H_{v,w}(q)\}$ defined for pairs of permutations
$v,w$ in the symmetric group $S_n$ (or more generally, any Weyl group $W$). The former family consists of
the celebrated \emph{Kazhdan-Lusztig polynomials},
which were introduced in \cite{Kazhdan.Lusztig}
to study representations of Hecke algebras.
There it was
conjectured that $P_{v,w}(q)\in {\mathbb Z}_{\geq 0}[q]$. This was later established \cite{KL:proof}
by interpreting $P_{v,w}(q)$ as the Poincar\'{e}
polynomial for Goresky-MacPherson's local intersection cohomology for the
torus fixed point $e_v$ of the Schubert variety $X_w$ in the complete flag variety ${\rm Flags}({\mathbb C}^n)$.

A key contribution to the theory is R.~Irving's theorem \cite{Irving} that the $P_{v,w}(q)$ are
{\bf upper semicontinuous}:
if $v'\leq v \leq w$
in Bruhat order, then $P_{v,w}(q)\preceq P_{v',w}(q)$, where ``$\preceq$''
means that, for each $i$,  the coefficient of $q^i$ in $P_{v,w}(q)$ is
weakly smaller than the coefficient of $q^i$ in $P_{v',w}(q)$. Thus, the Kazhdan-Lusztig polynomials are measures of the
singularities of Schubert varieties whose coefficient growth tracks
the worsening pathology of singularities as one moves along torus invariant ${\mathbb P}^{1}$'s towards the ``most singular'' point
$e_{id}\in X_w$. In particular, $P_{v,w}(q)=1$ if and
only if $e_v\in X_w$ is a (rationally) smooth point.

Conversely, the desire for insight into the combinatorics of Kazhdan-Lusztig polynomials
naturally leads to the basic problem of understanding where and how the singularities
of Schubert varieties worsen.
In view of this converse problem, the growth of any semicontinuous singularity
measure of Schubert varieties is of interest. One seeks concrete
comparisons of different measures; see, e.g., \cite{WYII} and the references therein.

Specifically, a well-studied semicontinuous measure is given by the {\bf Hilbert-Samuel
multiplicity} ${\rm mult}_{e_v}(X_w)$. However, while this
contains useful local data about $X_w$, even more is carried by the
${\mathbb Z}$-graded Hilbert series of ${\rm gr}_{\mathfrak{m}_{e_v}}{\mathcal O}_{e_v,X_w}$,
the associated graded ring of the local ring ${\mathcal O}_{e_v,X_w}$,
\[{\rm Hilb}({\rm gr}_{{\mathfrak m}_{e_v}}{\mathcal O}_{e_v,X_w},q)=\frac{H_{v,w}(q)}{(1-q)^{\ell(w)}},\]
where $\ell(w)={\rm dim}(X_w)$ is the {\bf Coxeter length} of $w$. In particular, ${\rm mult}_{e_v}(X_w)=H_{v,w}(1)$.

Conjecturally, each {\bf $h$-polynomial} $H_{v,w}(q)$ is also in ${\mathbb Z}_{\geq 0}[q]$, and moreover
is upper semicontinuous, just as is the case for Kazhdan-Lusztig polynomials.
These conjectures suggest that
the growth of the coefficients of the two families of polynomials is somehow correlated.
In this paper, we offer an examination in the Grassmannian case, and more generally in the case of
covexillary Schubert varieties inside ${\rm Flags}({\mathbb C}^n)$. There
the nonnegativity and semicontinuity conjectures are proved by deriving a new combinatorial rule for $H_{v,w}(q)$.
In addition, by introducing \emph{drift configurations} as a model for the Kazhdan-Lusztig polynomials
in these settings (after \cite{LS:KL} and \cite{Lascoux}), we
prove the inequality $P_{v,w}(q)\preceq H_{v,w}(q)$. This combinatorial discovery further indicates the link
between the two families; no alternative explanation via algebraic or geometric methods seems available at present.

Summarizing, the main thesis of this paper is that there exists
a parallel between $\{P_{v,w}(q)\}$ and $\{H_{v,w}(q)\}$. Our basis
for this perspective comes from proofs of compatible and
positive combinatorial rules for the two families of polynomials.

\subsection{Statements of the main conjecture and theorems}
Recapitulating, this paper formulates, and constructs supporting combinatorics for, the
following conjecture:
\begin{conjecture}
\label{conj:mainintro}
The $h$-polynomials $H_{v,w}(q)$ have nonnegative integral coefficients.
In addition, they are upper semicontinuous, i.e., if $v'\leq v$ in Bruhat order then
$H_{v,w}(q)\preceq H_{v',w}(q)$.
\end{conjecture}

The nonnegativity claim would actually be immediate
if ${\rm gr}_{{\mathfrak m}_{e_v}}{\mathcal O}_{e_v,X_w}$ is Cohen-Macaulay (see Section~2.2).
However, this latter assertion seems to be a folklore conjecture. Although ${\mathcal O}_{e_v,X_w}$ is
itself Cohen-Macaulay \cite{Raman},
this property might be lost when degenerating to ${\rm gr}_{{\mathfrak m}_{e_v}}{\mathcal O}_{e_v,X_w}$. On the other hand,
the results detailed in this paper and in \cite{Li.Yong} also support the
Cohen-Macaulayness conjecture.
In particular, it would follow from the stronger claim
\cite[Conjecture~8.5]{Li.Yong} asserting the vertex decomposability of Stanley-Reisner
simplicial complexes of certain Gr\"{o}bner degenerations of Kazhdan-Lusztig varieties.

The semicontinuity claim is itself a strengthening of the nonnegativity claim
since the smoothness of $X_w$ at $e_w$ implies $H_{w,w}(q)=1$. Furthermore, although the betti numbers of ${\rm gr}_{{\mathfrak
m}_{e_v}}{\mathcal O}_{e_v,X_w}$
are semicontinuous, the coefficients of $H_{v,w}(q)$ are an involved, signed expression in terms of those numbers. Therefore, this semicontinuity phenomenon seems substantive.

The natural projection map
\[\pi: {\rm Flags}({\mathbb C}^n) \twoheadrightarrow  {\rm Gr}_{k}({\mathbb C}^n): \ \ \
(\langle 0\rangle\subset F_1\subset\cdots\subset F_k\subset\cdots\subset F_{n-1}\subset {\mathbb C}^n)
\mapsto F_k,\]
where ${\rm Gr}_{k}({\mathbb C}^n)$ is the Grassmannian of $k$-dimensional planes
in ${\mathbb C}^n$, is a fibration: local properties of torus
fixed points $e_{\mu}\in X_{\lambda}\subseteq {\rm Gr}_{k}({\mathbb C}^n)$
for Young diagrams $\lambda,\mu\subseteq k\times (n-k)$,
are equivalent to local properties of $e_v\in X_w\subseteq {\rm Flags}({\mathbb C}^n)$ where
$v,w\in S_n$ are maximal  Coxeter length representatives of $\lambda,\mu$ where
the latter are thought of as cosets of $S_n/(S_{k}\times S_{n-k})$; see, e.g.,
\cite[Example~1.2.3]{Brion}. These
$v$ and $w$ are {\bf cograssmannian}, i.e., they have a unique ascent, at position $k$: $v(k)<v(k+1)$ and $w(k)<w(k+1)$.

Lifting Grassmannian problems to ${\rm Flags}({\mathbb C}^n)$ has the advantage of allowing one
to embed them within the wider class of {\bf covexillary Schubert varieties} $X_w$, i.e., where
$w$ is $3412$-avoiding: there are no indices $i_1<i_2<i_3<i_4$ such that
$w(i_1), w(i_2), w(i_3), w(i_4)$ are in the same relative order as $3412$.
This class appears more tractable than general flag Schubert varieties
since it shares many of the same features as Grassmannian Schubert varieties. However,
there is a salient difference: Grassmannian Schubert varieties are locally defined by equations that are homogeneous
with respect to the standard grading that assigns each variable degree one. In general, this is not true in the covexillary case.
This homogeneity means that taking associated graded of the local ring essentially does nothing,
and so ${\rm gr}_{{\mathfrak m}_{e_v}}{\mathcal O}_{e_v,X_w}$ is automatically Cohen-Macaulay;
see, e.g., \cite[Section~1]{Li.Yong} and Section~2.2.

The covexillary condition has already attracted significant attention; see, e.g., \cite{Lakshmibai.Sandhya, Lascoux, Manivel, KM:annals, KMY,
KMY:crelle, Li.Yong} and the references therein.
In particular, \cite[Section~2.4]{KM:annals}
connects the condition to \emph{ladder determinantal ideals} studied in commutative algebra.
Our three main theorems below concern the covexillary setting, providing our main cases of support
towards
both our main thesis and Conjecture~\ref{conj:mainintro}.

One of our results is to prove the following link between $H_{v,w}(q)$ and $P_{v,w}(q)$:
\begin{theorem}
\label{thm:mainineq}
For $w$ covexillary, $P_{v,w}(q)\preceq H_{v,w}(q)$ and $\deg  P_{v,w}(q)=\deg  H_{v,w}(q)$.
\end{theorem}
While the Grassmannian case \emph{per se} is new and supports our thesis,
the covexillary generality also further highlights the amenability of covexillary Schubert varieties.
Our proof of Theorem~\ref{thm:mainineq} is based on a new formula for covexillary Kazhdan-Lusztig polynomials.
An earlier rule was given by
A.~Lascoux \cite{Lascoux}, generalizing his earlier Grassmannian rule
with M.-P.~Sch\"{u}tzenberger \cite{LS:KL} (for more recent treatments of the Grassmannian case see, e.g.,
\cite{Zinn, Jones.Woo}).
Our formulation of a covexillary rule is in terms of \emph{drift configurations}. It is
entirely graphical and is perhaps more handy to compute.

To state our rule we use standard combinatorics of the symmetric group, see, e.g., \cite[Chapter~2]{Manivel}
as well as some terminology introduced in \cite{Li.Yong} (the reader may wish to compare
Examples~\ref{exa:drift} and~\ref{exa:433} below with what follows). Let $w\in S_n$ be covexillary.
Superimpose the {\bf graph} $G(v)$ of $v$ drawn with dots $\circ$ in positions $(n-w(j)+1,j)$
on top of the {\bf diagram}
\[D(w)=\{(i,j): i<n-w(j)+1, \textrm{ and } j<w^{-1}(n-i+1)\}\subset [n]\times [n].\]
Throughout, we use the convention that rows are indexed from bottom to top, and columns are indexed from left to right.
Move each box ${\mathfrak e}$ of the {\bf essential set}
\[\Ess(w)=\{(i,j)\in D(w): (i+1,j), (i,j+1)\notin D(w)\}\]
diagonally southwest by the number of dots of $G(v)$ weakly southwest of ${\mathfrak e}$.
Call the resulting boxes $\{{\mathfrak e}'\}$, and define
$B(v,w)$ to be the smallest Young diagram that contains
$\{{\mathfrak e}'\}$ and $(1,1)$ (we use French convention for our Young diagrams).
The {\bf shape} $\lambda(w)$ of $w$ is obtained by sorting the
vector counting the number of boxes in nonempty rows of $D(w)$ into decreasing order.
Now, draw $\lambda(w)$ in the southwest
corner of $B(v,w)$.

Declare that any corner of $\lambda(w)$ is $0$-{\bf special}.
Let ${\rm arm}(b)$ (respectively, ${\rm leg}(b)$)
refer to the boxes in $\lambda(w)$ strictly to the right (above) of $b$ and in the same row (column).
Inductively, a box $b\in\lambda(w)$ is
{\bf $z$-special}, for $z\in {\mathbb N}$ if it is maximally northeast subject to
\begin{itemize}
\item $|{\rm leg}(b)|=|{\rm arm}(b)|$; and
\item none of the boxes of $\{b\}\cup{\rm arm}(b)\cup {\rm leg}(b)$ are $y$-special for any $y< z$.
\end{itemize}
A box is ${\bf special}$ if it is $z$-special for some $z$.
The {\bf continent} of a special box $b$ is the set of $x\in\lambda(w)$ such that
$b$ is the maximally northeast special box that is weakly southwest of $x$.
%(Each continent is therefore itself a Young diagram.)
The union of continents is ${\rm Pangaea}(v,w)\subseteq \lambda(w)$ (the set difference being an immovable reference
continent).\footnote{As in the supercontinent that has been hypothesized to exist
$250$ million years ago in the theories of continental drift and plate tectonics}

\begin{definition}\label{def:drift} A {\bf drift configuration} ${\mathcal D}$
is a non-overlapping configuration of continents inside $B(v,w)$, such that
\begin{itemize}
\item each special box is diagonally weakly northeast
of its position in ${\rm Pangaea}(v,w)$; and
\item relative southwest-northeast positions of special cells are maintained.
\end{itemize}
\end{definition}
Let ${\rm drift}(v,w)$ be the set of all such ${\mathcal D}$ and let ${\rm wt}({\mathcal D})$
be the total distance traveled by the continents from ${\rm Pangaea}(v,w)$. Consider the generating series
\[Q_{v,w}(q)=\sum_{{\mathcal D}\in {\rm drift}(v,w)}q^{{\rm wt}({\mathcal D})}.\]

\begin{theorem}
\label{thm:main}
If $v,w\in S_n$ and $w$ is covexillary then:
\begin{itemize}
\item[(I)] $P_{v,w}(q)=Q_{v,w}(q)$.
\item[(II)] If we instead take every box of $\lambda(w)$ to be a separate ``country'', each of which ``drifts'' according to the
 rules of Definition~\ref{def:drift}, the total number of
drift configurations is ${\rm mult}_{e_v}(X_w)$; hence $P_{v,w}(1)\leq {\rm mult}_{e_v}(X_w)$ is manifest from {\rm (I)}.
\item[(III)] There is a vertex decomposable (thus shellable) simplicial complex ${\rm KL}_{v,w}$ that is homeomorphic to a ball or a
    sphere, and whose facets are labeled by ${\mathcal D}\in {\rm drift}(v,w)$.
\end{itemize}
\end{theorem}

Our proof of (I) is a bijection with A.~Lascoux's rule (which descends to a bijection with
the rule of \cite{LS:KL} for Grassmannians).
The multiplicity rule from (II)
just restates the theorem from \cite{Li.Yong} (cf. the Grassmannian rule of \cite{Naruse}).
Although the inequality of (II) is a consequence of Theorem~\ref{thm:mainineq}, we are emphasizing that
our rule from (I) is
compatible with our multiplicity rule and makes the inequality transparent.
Actually, whether such an inequality might exist was first asked
to us (independently) by S.~Billey and A.~Woo. Afterwards, H.~Naruse informed us that he has a
proof for all cominuscule $G/P$. These questions and results
provided us initial motivation for our work towards Theorem~\ref{thm:main}.
Note that as with the more general inequality of Theorem~\ref{thm:mainineq}, this inequality is not
true in general. For example, $P_{13425,34512}(1)=3$ while ${\rm mult}_{e_{13425}}(X_{34512})=2$.

The statement (III) is derived from \cite{KMY}. It points out a further resemblance to the
combinatorics of ${\rm mult}_{e_v}(X_w)$ in \cite{Li.Yong}, where a similar complex
also appears.
\begin{example}\label{exa:drift}
Figure~\ref{fig:1} depicts ${\rm Pangaea}(v,w)$ with six (colored) continents where
$$w=\underline{20}\;\underline{19}\;\underline{18}\;\underline{11}\;\underline{10} \;9 \;8\; \underline{12}\;
\underline{17}\;\underline{16}\;7\;6\;\underline{15}\;\underline{14}\;\underline{13}\;5\;4\;3\;2\;1,  \mbox{\ and $v=id$.\ \ \ \ \ \ \ \ \ \ \ \ \ \ \  \ \ \ \ \  \ \ \ \ \ \ \ \ \qed}$$
  \begin{figure}[h]%--means top, bottom, here.
    \centering % optional - centre figure
    \includegraphics[width=0.80\columnwidth]{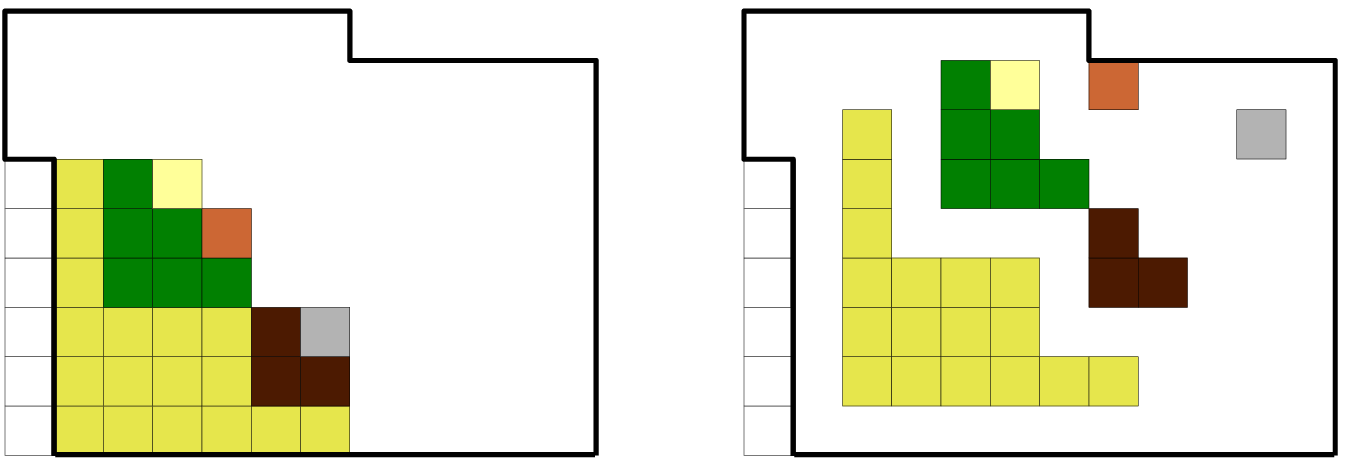}
\begin{picture}(1,1)
\put(-300,60){${\rm Pangaea}(v,w)$}
\put(-30,60){${\mathcal D}$}
\end{picture}
\caption{\label{fig:1} ${\rm Pangaea}(v,w)$ and a particular ${\mathcal D}\in {\rm drift}(v,w)$; ${\rm wt}({\mathcal D})=14$}
    \end{figure}
\end{example}

\begin{example}
\label{exa:433}
Let $w=\underline{10}954382761$, $v=234651789\underline{10}$. Here $\lambda(w)=(4,4,3)$. Starting from
$D(w)$, and the overlaid dots $\circ$ of $G(v)$, we derive $B(v,w)$. The special boxes are marked by $+$'s. See Figure~\ref{fig:2}.
  \begin{figure}[h]%--means top, bottom, here.
    \centering % optional - centre figure
    \includegraphics[width=.3\columnwidth]{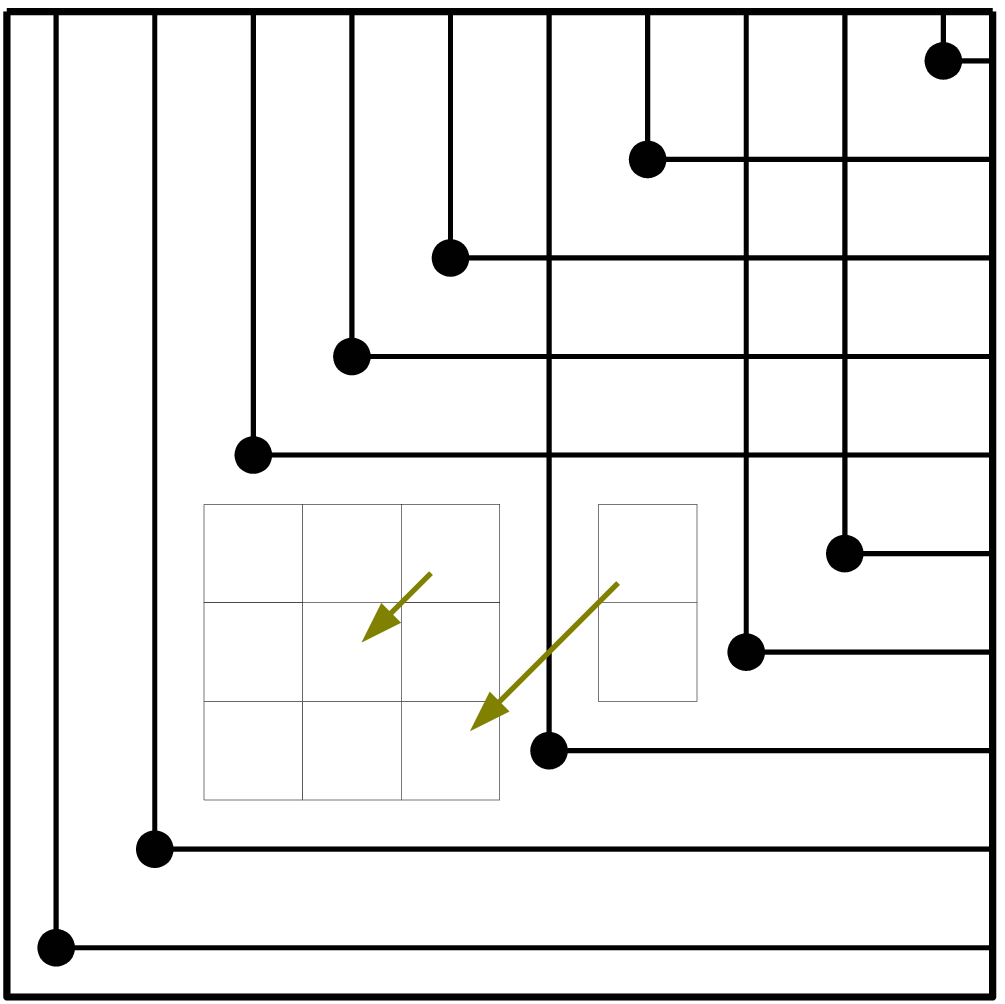}
\begin{picture}(1,1)
\put(-70,130){$\circ$}
\put(-140,115){$\circ$}
\put(-126,101){$\circ$}
\put(-112,87){$\circ$}
\put(-84,73){$\circ$}
\put(-98,59){$\circ$}
\put(-56,45){$\circ$}
\put(-42,31){$\circ$}
\put(-28,18){$\circ$}
\put(-14,5){$\circ$}
\put(-84,62){${\mathfrak e}_1$}
\put(-57,62){${\mathfrak e}_2$}
\put(-102,45){${\mathfrak e}'_1$}
\put(-88,32){${\mathfrak e}'_2$}
\end{picture}
\quad
\quad
\quad
\quad
\quad
  \includegraphics[width=.2\columnwidth]{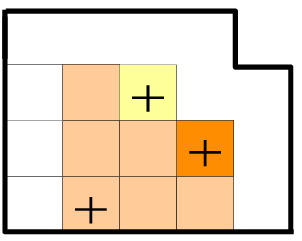}
\begin{picture}(1,1)
\put(-30,80){$B(v,w)$}
\put(-90,25){$\lambda(w)$}
\put(-140,30){$\implies$}
\end{picture}
\caption{\label{fig:2} An overlay of $D(w)$ with $G(w)$ ($\bullet$'s) and $G(v)$ ($\circ$'s); constructing $B(v,w)$}
    \end{figure}
Now $\Ess(w)=\{{\mathfrak e}_1, {\mathfrak e}_2\}$ (being the maximally northeast boxes of each
connected component of $D(w)$) move to
$\{{\mathfrak e}_1', {\mathfrak e}_2'\}$, as determined by the $\circ$'s of $G(v)$. The five drift configurations
are shown in Figure 3.
  \begin{figure}[h]%--means top, bottom, here.
    \centering % optional - centre figure
    \includegraphics[width=1\columnwidth]{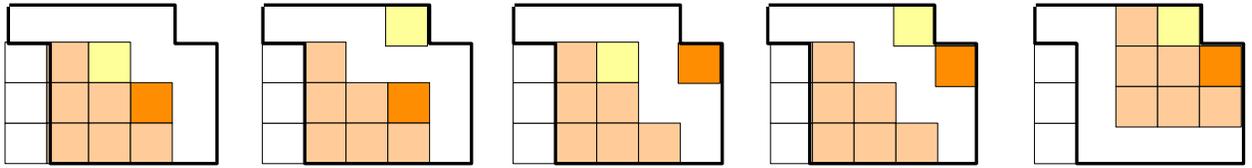}
\caption{\label{figure3} Drift configurations for $Q_{234651789\underline{10},\underline{10}954382761} (q)=1+2q+q^2+q^3$}
    \end{figure}
    \qed
\end{example}
Our proof of Theorem~\ref{thm:mainineq} also depends on a new (and the first manifestly
positive)
combinatorial rule for covexillary $H_{v,w}(q)$. It additionally implies special cases of the nonnegativity and upper semicontinuity
conjectures. Identify a partition
$\lambda=(\lambda_1\geq\cdots\geq\lambda_\ell>0)$ with its Young diagram (in French notation). Recall,
a Young tableau $T$ of shape $\lambda$ is {\bf semistandard} if it is weakly increasing along rows and strictly increasing up columns.
Given a vector $\textbf{b}=(b_1,\dots,b_\ell)$, we say $T$ is {\bf flagged} by ${\bf b}$ if each entry in row $i$ is
at most $b_i$. Let ${\rm SSYT}(\lambda,{\bf b})$ denote the set of semistandard Young tableaux flagged by ${\bf b}$.
A (nonempty) set-valued filling is semistandard
if each tableau obtained by choosing a singleton from each set
gives a semistandard tableaux in the above sense \cite{Buch:KLR}.
Similarly, we define {\bf flagged set-valued semistandard tableaux}, and the set
${\rm SetSSYT}(\lambda,{\bf b})$ \cite{KMY}.

Define $U\in {\rm SetSSYT}(\lambda,{\bf b})$ to be {\bf lower saturated} if no smaller number can be added to
any box $U(i,j)$ while maintaining semistandardness, i.e., in symbols, each
$$U(i,j)=[\alpha,\beta]:=\{\alpha, \alpha+1,\dots,\beta-1,\beta\},$$
for some $\alpha,\beta$ (depending on $i,j$) where
$$\alpha=\max\big{\{}\max U(i,j-1), 1+\max U(i-1,j)\big{\}}.$$
Our convention for lower saturated tableaux is that  $U(i,0)=1$ for all  $i>0$ and $U(0,j)=0$ for all $j>0$.
Let
\[{\rm Lower}(\lambda,{\bf b})\subseteq {\rm SetSSYT}(\lambda,{\bf b})\]
denote this subset of lower saturated tableaux.

Define the {\bf saturation} ${\rm sat}(T)\in {\rm Lower}(\lambda,{\bf b})$
of $T\in {\rm SSYT}(\lambda,{\bf b})$ to be
$${\rm sat}(T)(i,j)=[\max\{T(i,j-1), 1+T(i-1,j)\}, T(i,j)].$$

For $U\in {\rm SetSSYT}(\lambda,{\bf b})$, let
\[\ex(U)=|U|-|\lambda|,\]
where $|U|$ refers to the number of entries of $U$ and $|\lambda|=\lambda_1+\lambda_2+\cdots$.

Finally, if $T\in {\rm SSYT}(\lambda,{\bf b})$ set
\begin{equation}
\label{eqn:depth}
\ey(T):=\ex(\sat(T))=|\sat(T)|-|T|.
\end{equation}

If $\lambda(w)=(\lambda(w)_1\ge\cdots\geq\lambda(w)_\ell>0)$ then define ${\bf b}={\bf b}(\Theta_{v,w})=(b_1,\dots,b_\ell)$ by
$$b_i=\max\{m:\; B(v,w)_m\ge \lambda(w)_i+m-i\}.$$
This is the maximum distance that the rightmost box in row $i$ can drift diagonally northeast within $B(v,w)$ (ignoring presence
of other boxes).
\begin{theorem}
\label{thm:Kvwrule}
Let $w\in S_n$ be covexillary. Then
$$H_{v,w}(q)=\sum_{T\in {\rm SSYT}(\lambda(w),{\bf b}(\Theta_{v,w}))}q^{\ey(T)}=\sum_{U\in {\rm Lower}(\lambda(w),{\bf
b}(\Theta_{v,w}))}q^{\ex(U)}.$$
Moreover, Conjecture~\ref{conj:mainintro} is true under the hypothesis.
\end{theorem}

\begin{example}
 \label{exa:a}
 For $n=5$, $w=52341$, $v=12345$. There are five semistandard tableaux of shape $(2,1)$ and flagged by $(2,3)$:
$$\ktableau{{2}\\{1}&{1}}\quad
\ktableau{{3}\\{1}&{1}}\quad
\ktableau{{2}\\{1}&{2}}\quad
\ktableau{{3}\\{1}&{2}}\quad
\ktableau{{3}\\{2}&{2}}
$$
Their saturations are:
$$\ktableau{{2}\\{1}&{1}}\quad
\ktableau{{2,3}\\{1}&{1}}\quad
\ktableau{{2}\\{1}&{1,2}}\quad
\ktableau{{2,3}\\{1}&{1,2}}\quad
\ktableau{{3}\\{1,2}&{2}}
$$
The corresponding $\ex$ values are:
$$0, 1, 1, 2, 1.$$
Thus by Theorem~\ref{thm:Kvwrule}, $H_{v,w}(q)=1+3q+q^2.$\qed
\end{example}

\begin{example}
Continuing Example~\ref{exa:a}, there are four drift configurations of the two continents,
\begin{figure}[h]%--means top, bottom, here.
  \centering % optional - centre figure
  \includegraphics[width=.5\columnwidth]{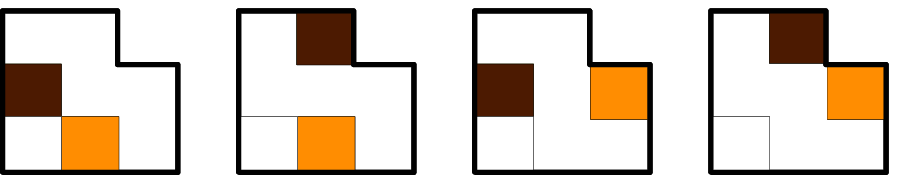}
\end{figure}

%Their image under map $g$ is
%
%$$\ktableau{{2}\\{1}&{1}}\quad
%\ktableau{{3}\\{1}&{1}}\quad
%\ktableau{{2}\\{1}&{2}}\quad
%\ktableau{{3}\\{1}&{2}}
%$$
The Kazhdan-Lusztig polynomial $P_{v,w}(q)=1+2q+q^2$. We see that $P_{v,w}(q)\preceq H_{v,w}(q)$, in agreement with
Theorem~\ref{thm:mainineq}.\qed
\end{example}

\subsection{Organization and contents}
In Section~2, we state some preliminaries and further discuss Conjecture~\ref{conj:mainintro}.
 We then prove Theorem~\ref{thm:Kvwrule}. In Section~3, we briefly
recall, for comparison, basics about Kazhdan-Lusztig theory.
We then prove Theorem~\ref{thm:mainineq} while temporarily assuming Theorem~\ref{thm:main}(I).
Section~4 is devoted to the construction of
the simplicial complex of Theorem~\ref{thm:main}(II) and proof of its asserted properties.
We furthermore define polynomials
generalizing $Q_{v,w}(q)$ that naturally arise from this complex. In Section~5
we prove Theorem~\ref{thm:main}(I). We end that section with two comments (Remarks~\ref{remark:twosymmetries} and~\ref{remark:anotherbound}) about further properties of $P_{v,w}(q)$ that can be
deduced from the rule.
In Section~6, we give a formula for a different ``$q$-analogue'' of
${\rm mult}_{e_v}(X_w)$ than $H_{v,w}(q)$.
In Section~7, we offer some final remarks.

%%%%%%%%%%%%%%%%%%%%%%%%%%%%%%%%%%%%%%%%%%%%%%%%%%%%%%%%%%%%%%%%%%%%%%%%

\section{Hilbert series of the local ring ${\mathcal O}_{e_v,X_w}$}

\subsection{Preliminaries}
We use the usual identification ${\rm Flags}({\mathbb C}^n)=GL_n/B$ where $B$ is the
Borel subgroup consisting of invertible upper triangular matrices. Thus $GL_n$ acts on ${\rm Flags}({\mathbb C}^n)$ by left
multiplication, as does $B$, and the torus $T$ of invertible diagonal matrices.
For each $v\in S_n$, let $e_v$ denote the associated $T$-fixed point.
The {\bf Schubert cell} $X_{w}^{\circ}:=Be_w$ while
its Zariski closure is the {\bf Schubert variety}
$X_w={\overline{X_{w}^{\circ}}}$, an irreducible variety of dimension $\ell(w)$.
We have that $e_v\in X_w$ if and only if $v\leq w$ in {\bf Bruhat order}. A neighborhood of each point
$p\in X_w$ is isomorphic to a neighborhood of some $e_v$, by the action of $B$.
Hence, it suffices to restrict attention to
$T$-fixed points. Let $B_{-}$ be the opposite Borel subgroup of invertible lower triangular matrices.
If we set $\Omega_{v}^{\circ}=B_{-}vB/B$ to be the {\bf opposite Schubert cell}, then
up to crossing by affine space, a local neighbourhood of $e_v\in X_w$ is given by the {\bf Kazhdan-Lusztig variety}
${\mathcal N}_{v,w}=X_w\cap \Omega_{v}^{\circ}$ \cite[Lemma~A.4]{Kazhdan.Lusztig}.

Suppose $p$ is a point on a scheme $Y$. Let ${\rm gr}_{{\mathfrak m}_{p}}{\mathcal O}_{p,Y}$
denote the {\bf associated graded ring} of the local ring
${\mathcal O}_{p,Y}$ with respect to its maximal ideal ${\mathfrak m}_p$, i.e.,
\[{\rm gr}_{{\mathfrak m}_{p}}{\mathcal O}_{p,Y}=\bigoplus_{i\geq 0} m_{p}^i/m_p^{i+1}.\]
Since ${\rm gr}_{{\mathfrak m}_{p}}{\mathcal O}_{p,Y}$ picks up a ${\mathbb Z}$-grading,
it now makes sense to discuss its Hilbert series. One can always
express this series in the form
\[{\rm Hilb}({\rm gr}_{{\mathfrak m}_p}{\mathcal O}_{p,Y},q)=\frac{H_{p,Y}(q)}{(1-q)^{\dim Y}}\]
where $H_{p,Y}(q)\in {\mathbb Z}[q]$ is the {\bf $h$-polynomial} associated to $p\in Y$.
It follows from standard facts that $H_{p,Y}(1)={\rm mult}_{p}(Y)$;
see, e.g., \cite[Theorem 5.4.15]{Kreuzer.Robbiano}. Hence $H_{p,Y}(q)=1$ if and only if $Y$ is smooth at $p$. In addition, note
$H_{p,Y}(0)=1$, since this is the
dimension of the zero graded piece of ${\rm gr}_{{\mathfrak m}_p}{\mathcal O}_{p,Y}$, i.e., the dimension of the field
${\mathcal O}_{p,Y}/{\mathfrak m}_p$.

Now, for any $v,w\in S_n$, we define $H_{v,w}(q)\in\mathbb{Z}[q]$ to be the $h$-polynomial associated to
$e_v\in X_w$. At present, there is no purely combinatorial formula (even non-positive or recursive) for
computing $H_{v,w}(q)$. However, instead one can utilize the explicit coordinates and equations for the ideal $I_{v,w}$
 to define ${\mathcal N}_{v,w}={\rm Spec}\left({\mathbb C}[z^{(v)}]/I_{v,w}\right)$, as done in \cite[Section~3.2]{WYII}.
  Then one can Gr\"{o}bner degenerate ${\mathcal N}_{v,w}$ to a scheme theoretic union
of coordinate subspaces ${\mathcal N}_{v,w}'$,
using any of the term orders $\prec_{v,w,\pi}$ from \cite[Section~3]{Li.Yong}. As explained in Theorem~3.1 (and its proof)
of \cite{Li.Yong}, the stated Gr\"{o}bner degenerations degenerate not only ${\mathcal N}_{v,w}$ but also its projectivized
tangent cone ${\rm Proj}({\rm gr}_{{\mathfrak m}_{e_v}}{\mathcal O}_{e_v, X_w})$.
Therefore the $h$-polynomial of ${\mathcal N}_{v,w}'$
equals $H_{v,w}(q)$.

\subsection{Conjectures}
Let us now return to the discussion of Conjecture~\ref{conj:mainintro}. Using the method
for computing $H_{v,w}(q)$ summarized above, we obtained exhaustive checks for $n\leq 7$ of the following claim, restated
from the introduction:
\begin{nonnegativityconjecture}
$H_{v,w}(q)\in {\mathbb Z}_{\geq 0}[q]$.
\end{nonnegativityconjecture}

In \cite[Conjecture~8.5]{Li.Yong} we conjectured that within the family of term orders $\prec_{v,w,\pi}$, at least one gives a
Gr\"{o}bner
limit scheme ${\mathcal N}_{v,w}'$ that is reduced, equidimensional and
whose Stanley-Reisner simplicial complex $\Delta_{v,w}$ is a vertex-decomposable ball or sphere.
This implies in particular that $\Delta_{v,w}$ is shellable and thus Cohen-Macaulay. If this conjecture were true, it would follow that
${\rm gr}_{\mathfrak{m}_{e_v}}{\mathcal O}_{e_v,X_w}$ is Cohen-Macaulay. Thus the nonnegativity Conjecture
would hold by, e.g., \cite[Corollary~4.1.10]{Bruns.Herzog}.

In the case that $I_{v,w}$ is a homogeneous ideal, with respect to the standard
grading that assigns each variable degree $1$, since ${\mathcal O}_{e_v,X_w}$ is Cohen-Macaulay \cite{Raman},
it follows that the associated graded ring is Cohen-Macaulay; see e.g., \cite[Exercise~2.1.27(c)]{Bruns.Herzog}. Hence nonnegativity
follows in this case. A.~Knutson has shown that this homogeneity occurs whenever $w$ is $321$-avoiding \cite[pg.~25]{Knutson:frob}.
Moreover,
in \cite[Section~5]{WYIII} it was explained how ``parabolic moving'' reduces
a large percentage of cases (for $n\leq 10$) to the homogeneous case. However, not every case can be so reduced, including those
in the covexillary class. Thus, these cases provide further support
for the above conjecture, separate from Theorem~\ref{thm:Kvwrule}.

\begin{upperconjecture}
If $v'\le v \leq w$ in Bruhat order, then $H_{v,w}(q)\preceq H_{v',w}(q)$.
\end{upperconjecture}

Unfortunately, even if we knew ${\rm gr}_{\mathfrak{m}_{e_v}}{\mathcal O}_{e_v,X_w}$ to be Cohen-Macaulay, we do not know any
way to express these coefficients in homological terms that would make the upper semicontinuity
conjecture transparent. It should be noted that the proof of this property for Kazhdan-Lusztig polynomials in \cite{Irving}
was not achieved using the geometry of Schubert varieties. However, see the geometric argument for the more general result
\cite[Theorem~3.6]{Braden}.

Although any proof of the above conjectures is desired, ideally one would also like combinatorial explanations
of the properties.

Let us pause to collect some further facts for small $n$ in the following computational result.
For (D) below we refer the reader to \cite[Section~2.1]{WYII} for the
definition of \emph{interval pattern avoidance} of $[x,y]\in S_\infty\times S_{\infty}$. There we explain that the existence
of an \emph{interval pattern embedding} guarantees ${\mathcal N}_{x,y}\cong {\mathcal N}_{\tilde w,w}$,
where $[x,y]\cong [\tilde w,w]$ is an isomorphism of posets of Bruhat intervals in $S_{\infty}$. Thus, if the inequality
$P_{x,y}(q)\preceq H_{x,y}(q)$ fails, so must $P_{\tilde w,w}(q)\preceq H_{\tilde w,w}(q)$.

\begin{proposition}
\label{prop:less6}
\begin{itemize}
\item[(A)] $\deg H_{v,w}(q)\le \deg P_{v,w}(q)$ for $v\leq w\in S_n$ and $n\leq 6$.
\item[(B)] $\deg H_{v,w}(q)\le \frac{\ell(w)-\ell(v)-1}{2}$ for $v<w\in S_n$ and $n\leq 7$.
\item[(C)] The coefficients of $H_{v,w}(q)$ form a unimodal sequence for $v,w\in S_n$ and  $n\leq 7$.
\item[(D)] $P_{v,w}(q)\preceq H_{v,w}(q)$ holds for all $v\leq w\in S_n$ and $n\leq 6$, if and
only if $w$ interval pattern avoids
\[\begin{array}{c}
${\rm [14235,45123], [31524,53412], [14325,45312]}$\\
${\rm [13425,34512],  [24153,45231],  [154326,564312]}.$
\end{array}\]
(Note that the first and fourth intervals, and the second and fifth intervals are related by taking inverses.
For all $n\geq 1$, the inequality fails whenever $w$ contains one of these intervals.)
\end{itemize}
\end{proposition}
\noindent
\emph{Proof and discussion:} Each of the
assertions were verified using {\tt Macaulay~2}. For (A) and (B) note that
$\deg P_{v,w}(q)\leq \frac{\ell(w)-\ell(v)-1}{2}$
is a standard fact about Kazhdan-Lusztig polynomials; cf. Section~3.1.

For (D), computation shows that
$P_{v,w}(q)=H_{v,w}(q)$ for $n\leq 4$, so the inequality holds in that situation.
We checked that each of the intervals $[x,y]$
listed corresponds to a failure of the inequality for $n\leq 5$. For $n=6$ we computationally
verified the claim (there are $36$ cases $w\in S_6$ where the inequality fails for some $v\leq w$, and of those
only one cannot be blamed on the $n=5$ cases). The $n> 6$ case follows from general properties of interval
pattern embeddings recalled above.
\qed

One might conjecture that both (A) and its weak form (B) hold for all $n$. However with (A), experience has shown that
data for $n\leq 6$ is soft evidence for any conjecture that involves Kazhdan-Lusztig polynomials. Note that
if (A) is true, one cannot have
$P_{v,w}(q)\preceq H_{v,w}(q)$ unless $\deg H_{v,w}(q)= \deg P_{v,w}(q)$, which is indeed what we show when
$w$ is covexillary.

In view of (C), it is also natural to guess that unimodality is true in general.
One warning however
is that the stronger assertion that the coefficients of $H_{v,w}(q)$ are \emph{log-concave} is false, as the
example below shows:

\begin{example}
Let $w=5671234$, $v=1352476$, computation using {\tt Macaulay~2} shows there is a choice of $\prec_{v,w,\pi}$ such that
${\mathcal N}_{v,w}'$ is Cohen-Macaulay (but not Gorenstein), and that $H_{1352476,5671234}(q)=1+2q+q^2+q^3$, which is not
log-concave.\qed
\end{example}

By contrast, see the related work of M.~Rubey \cite{Rubey} that shows log-concavity holds in
a special ladder determinantal case (note that $w$ is not covexillary in our counterexample).

Even knowing Cohen-Macaulayness of
${\rm gr}_{\mathfrak{m}_{e_v}}{\mathcal O}_{e_v,X_w}$ does not, in and of itself, prove unimodality. In fact,
R.~Stanley had conjectured~\cite[Conjecture 4(a)]{Stanley:unimodal} unimodality for a
general graded Cohen-Macaulay domain $R$ over a field which is
generated by $R_1$. Actually, he even conjectured the stronger claim of log-concavity, although counterexamples to the stronger
claim were later found by G.~Niesi-L.~Robbiano, see \cite[Section~5]{Brenti}.
(The above example gives a different counterexample to Stanley's log-concavity conjecture.)

It should also be mentioned that in contrast, the Kazhdan-Lusztig polynomials are not in general unimodal
and in fact P.~Polo \cite{Polo} proved that every nonnegative integral polynomial with constant coefficient $1$ is some $P_{v,w}(q)$.

While Theorem~\ref{thm:Kvwrule} allows us to prove
the nonnegativity, upper-semicontinuity and degree properties
for covexillary $X_w$, a resolution to the following has alluded us:

\begin{problem}
Give a combinatorial proof (e.g., using Theorem~\ref{thm:Kvwrule}) for the unimodality conjecture,
when $w$ is covexillary (or even cograssmannian) by establishing a sequence of explicit
injections and surjections of the relevant
Young tableaux.
\end{problem}

Concerning (D), we do not expect the characterization to be valid for all $n$. Instead, one aims to
expand this list into a (human-readable) classification, via a finite list of families of patterns to avoid,
as is the case for many other properties studied in \cite{WYII}.

Using the analogy with Kazhdan-Lusztig theory, numerous further problems, that had been previously
considered for $P_{v,w}(q)$ but not $H_{v,w}(q)$, make sense. To name a few: Is $H_{v,w}(q)$
determined by the poset isomorphism class of the interval $[v,w]$ in Bruhat order? (This is an analogue
of a conjecture of G.~Lusztig.) Can one give a combinatorial algorithm for computing $H_{v,w}(q)$?
Better yet, can one find a positive combinatorial rule for $H_{v,w}(q)$, thus establishing the nonnegativity conjecture?

\subsection{Proof of Theorem~\ref{thm:Kvwrule}}
Continuing the definitions before the statement of Theorem~\ref{thm:Kvwrule} in Section~1,
set
$$\sup: {\rm SetSSYT}(\lambda,{\bf b})\to {\rm SSYT}(\lambda, {\bf b})$$ by sending $U$ to $T$ where $T(i,j)=\max U(i,j)$.

Clearly,
\begin{lemma}
\label{lemma:mutuallyinv}
The maps
\[\sat: {\rm SSYT}(\lambda,{\bf b})\to {\rm Lower}(\lambda,{\bf b})
\mbox{\ \ and \ \ }
\sup|_{{\rm Lower}(\lambda,{\bf b})}: {\rm Lower}(\lambda,{\bf b})\to {\rm SSYT}(\lambda,{\bf b})\]
are mutually inverse bijections.
\end{lemma}

Let us recall some definitions and terminology utilized in \cite{Li.Yong}.
Define $r^{w}_b=r^w_{(i,j)}$ to be the number of $\bullet$ of $G(w)$ weakly southwest of the box $b=(i,j)$.
Given $v\le w$ and $w$ covexillary, $\Theta_{v,w}\in S_n$ is defined \cite{Li.Yong} to be the unique permutation such
that $\lambda(w)=\lambda(\Theta_{v,w})$ and
$$\Ess(\Theta_{v,w})=\{{\mathfrak e}': \; {\mathfrak e}' \textrm{ is obtained by moving each ${\mathfrak e}\in \Ess(w)$ diagonally
southwest by $r^v_{\mathfrak e}$ units}\}.$$
The permutation $\Theta_{v,w}$ was proved to be itself covexillary.

Define $B(w)$ to be the smallest Young diagram
with southwest corner in position $(1,1)$ that contains all of $\Ess(w)$. Set
\[B(v,w)=B(\Theta_{v,w}).\]
If $\lambda(w)=(\lambda(w)_1\ge\cdots\geq\lambda(w)_\ell>0)$ then define ${\bf b}={\bf b}(w)=(b_1,\dots,b_\ell)$ by
$$b_i=\max\{m:\; B(w)_m\ge \lambda(w)_i+m-i\}.$$
The above agrees with, and slightly reformulates, the definitions of $B(v,w)$ and ${\bf b}$ from the introduction.

In \cite[Theorem 6.6]{Li.Yong} we proved that
\[{\rm Hilb}({\rm gr}_{{\mathfrak m}_{e_v}}{\mathcal O}_{e_v,X_w},q)=G_{\lambda(w)}(q)/(1-q)^{{n\choose 2}}\]
where
\[G_{\lambda(w)}(q)=\sum_{k\geq |\lambda(w)|}(-1)^{k-|\lambda(w)|}(1-q)^k\times \#{\rm SetSSYT}(\lambda(w),{\bf b},k)\]
and $\#{\rm SetSSYT}(\lambda(w),{\bf b},k)$ is the number of flagged set-valued semistandard Young tableaux
of shape $\lambda(w)$ with flag ${\bf b}={\bf b}(\Theta_{v,w})$ which use exactly $k$ entries.

Since the local ring $\mathcal{O}_{e_v,X_w}$ is of dimension $\ell(w)={n\choose 2}-|\lambda(w)|$, we rewrite
$${\rm Hilb}({\rm gr}_{{\mathfrak m}_{e_v}}{\mathcal O}_{e_v,X_w},q)=H_{v,w}(q)/(1-q)^{\ell(w)}$$
where
$$H_{v,w}(q)=\sum_{U\in {\rm SetSSYT}(\lambda(w),{\bf b})}(q-1)^{\ex(U)}.$$

We need to show that
\begin{equation}\label{eq:SetSSYT}
\sum_{U\in {\rm SetSSYT}(\lambda(w),{\bf b})}(q-1)^{\ex(U)}=\sum_{T\in {\rm SSYT}(\lambda(w),{\bf b})}q^{\ey(T)}
\end{equation}
by proving that, for every $T\in {\rm SSYT}(\lambda(w),{\bf b})$,
$$\sum_{U\in \sup^{-1}(T)}(q-1)^{\ex(U)}=q^{\ey(T)}.$$
There are $\ey(T)$ elements in $\sat(T)$ but not in $T$. We can delete any subset of those elements from $\sat(T)$ and obtain
$T'\in\sup^{-1}(T)$ (so $\#\sup^{-1}(T)=2^{\ey(T)}$). Hence the left hand side is equal to
$$(1+(q-1))^{\ey(T)}=q^{\ey(T)},$$
and therefore the equality (\ref{eq:SetSSYT}) follows. Thus, the first equality of the theorem holds and the second is clear from
Lemma~\ref{lemma:mutuallyinv}.

The nonnegativity claim is manifest from the combinatorial rule; however, let us also give
a geometric proof. In \cite{Li.Yong}
we proved that for covexillary $w$, ${\mathcal N}_{v,w}$ degenerates, under a choice of $\prec_{v,w,\pi}$ to a Cohen-Macualay
limit scheme ${\mathcal N}_{v,w}'$.
Hence, nonnegativity of $H_{v,w}(q)$ follows from \cite[Corollary~4.1.10]{Bruns.Herzog} and the discussion of Section~2.1.

For the upper semicontinuity claim, fix $w\in S_n$ and suppose $v'\le v\leq w$. Consider an essential box ${\mathfrak e}\in \Ess(w)$. In
the construction of $\Ess(\Theta_{v,w})$, the essential box ${\mathfrak e}$ is moved diagonally southwest by $r_{\mathfrak e}^v$ units.
Since $v'\le v$, a standard characterization of Bruhat order
shows $r_{\mathfrak e}^{v'}\le r_{\mathfrak e}^v$. Thus, each
essential box ${\mathfrak e}$ moves further southwest in to its position in $\Ess(\Theta_{v,w})$ than it does for $\Ess(\Theta_{v',w})$.
Therefore,
\[B(v,w)\subseteq B(v',w),\]
and hence,
\[{\bf b}(\Theta_{v,w})=(b_1,\dots,b_\ell)\le{\bf b}(\Theta_{v',w})=(b'_1,\dots,b'_\ell),\]
in the sense that $b_i\le b'_i$ for every $i$. Consequently, ${\rm SSYT}(\lambda,{\bf b})\subseteq {\rm SSYT}(\lambda,{\bf b}')$, which
clearly implies $H_{v,w}(q)\preceq H_{v',w}(q)$, as desired. \qed

\section{Kazhdan-Lusztig theory}

\subsection{The Hecke algebra}
Let $R={\mathbb Z}[q^{\frac{1}{2}}, q^{-\frac{1}{2}}]$ be the ring of Laurent polynomials over ${\mathbb Z}$ in the indeterminate
$q^{\frac{1}{2}}$. The {\bf Hecke algebra} $\mathcal H_{n-1}$ of $S_n$
is the algebra over $R$ with basis
$\{T_w: w\in S_n\}$ and relations

\begin{eqnarray}\nonumber
T_{s_i}T_w & = & T_{s_iw} \mbox{\ \ \ \ \ \ \ \ \ \ \ \ \ \ \ \ \ \ \ \ \ \ \  if $\ell(s_iw)>\ell(w)$}\\ \nonumber
T_{s_i}^2 & = & (q-1)T_{s_i}+qT_{id}.
\end{eqnarray}
There is an involution $\iota:{\mathcal H}_{n-1}\to {\mathcal H}_{n-1}$ defined by
$\iota(q^{\frac{1}{2}})=q^{-\frac{1}{2}}$ and $\iota(T_w)=T_{w^{-1}}^{-1}$.

\excise{The Kazhdan-Lusztig polynomial $P_{v,w}(q)$ can be defined recursively as follows. Define $P_{v,w}(q)=0$ if $v\nleq w$ and set
$P_{w,w}(q)=1$. For any simple reflection $s\in S_n$ such that $sw<w$, we have
$$P_{v,w}(q)=q^{1-c}P_{sv,sw}(q)+q^cP_{v,sw}(q)-\sum\mu(z,sw)q^{\frac{1}{2}(\ell(w)-\ell(z))}P_{v,z}(q),$$
were the sum is over all $z<sw$ for which $sz<z$. Here $\mu(z,sw)$ is the coefficient of $q^{\frac{1}{2}(\ell(sw)-\ell(z)-1)}$ in
$P_{z,sw}(q)$, and $c=0$ if $v<sv$ and $c=1$ otherwise. It is known that $P_{v,w}(q)\in\mathbb{N}[q]$, and if $v<w$ then $\deg
P_{v,w}(q)\le\frac{\ell(w)-\ell(v)-1}{2}$.}

It was proved in \cite{Kazhdan.Lusztig} that there exists a basis
$\{{\mathcal C}'_w\}$ of ${\mathcal H}_{n-1}$ that is uniquely determined by the conditions that
\[\iota({\mathcal C}'_w)={\mathcal C}'_w\]
and
\[C'_w=(q^{-\frac{1}{2}})^{\ell(w)}\sum_{v\le w} P_{v,w}(q) T_v,\]
where
\begin{itemize}
\item[(i)] $P_{w,w}(q)=1$;
\item[(ii)] $P_{v,w}(q)=0$ if $v\not\leq w$; and
\item[(iii)] $P_{v,w}(q)\in {\mathbb Z}[q]$ is of degree $\leq \frac{\ell(w)-\ell(v)-1}{2}$ if $v< w$.
\end{itemize}

The existence of this basis was established by an explicit recursion for the
{\bf Kazhdan-Lusztig polynomials} $P_{v,w}(q)$ which we omit.
Our source for these facts is \cite[Chapter 6]{Billey.Lakshmibai} where we refer the reader to for further details.

The conditions (i) and (ii) also hold for the $H_{v,w}(q)$, while (iii) conjecturally
holds (cf. Proposition~\ref{prop:less6} and the discussion thereafter).
It is mildly tempting to think about another basis of the Hecke algebra defined by replacing $P_{v,w}(q)$ by
$H_{v,w}(q)$ in the above definition of $C'_{w}$. While this other basis has a unimodular transition matrix
with the Kazhdan-Lusztig basis, it doesn't possess any of the other nice properties, such as positive structure constants,
or invariance under the involution $\iota$.

\subsection{Proof of Theorem~\ref{thm:mainineq}}
Recall that in what follows, we are assuming the formula for $P_{v,w}(q)$ from Theorem~\ref{thm:main} that
we prove in Section~5.

Given any box $(i,j)\in \lambda(w)$, let
$(\hat{i},j)$ be the top-most box in the column $j$.

Let ${\bf b}={\bf b}({\Theta}_{v,w})$, cf. just before Theorem~\ref{thm:Kvwrule}, or Section~2.3. Define
$$\Psi:{\rm drift}(v,w)\to {\rm SSYT}(\lambda(w),{\bf b})$$
by sending a drift configuration $\mathcal{D}$ to the semistandard tableau $T$, as follows.
For each special box $(i,j)\in\lambda(w)$ we fill $(\hat{i},j)$
with the entry $(\hat{i}+d)$, where $d$ is the distance moved in ${\mathcal D}$
by the continent associated to $(i,j)$, from ${\rm Pangaea}(v,w)$. Note that the value of this
entry is the height of the box $(\hat{i},j)$ after drifting in the drift configuration ${\mathcal D}$.
Now fill in the remaining
empty boxes of $\lambda(w)$ by working down columns, from right to left, according to the following prescription:
\begin{equation}
\label{eqn:prescription}
T(i,j)=\min\{T(i+1,j)-1, T(i-1,j+1)+1\}.
\end{equation}
By convention, set
\begin{equation}
\label{eqn:convention1}
T(i,j)=\infty \mbox{\ if $i>0$ and $(i,j)\notin \lambda(w)$, or if $j>m$;}
\end{equation}
and
\begin{equation}
\label{eqn:convention2}
T(i,j)=0 \mbox{\ if $i=0$ and $j\leq m$,}
\end{equation}
where $m$ is the number of columns in $\lambda(w)$.

\begin{example}
For the five drift configurations $\mathcal{D}$ in Example~\ref{exa:433} (see Figure~\ref{figure3}), the corresponding $\Psi(\mathcal{D})$
are as follows, where the boxes $(\hat{i},j)$ corresponding to special boxes are underlined.
$$\ktableau{{3}&{\underline{3}}&{\underline{3}}\\{2}&{2}&{2}&{\underline{2}}\\{1}&{1}&{1}&{1}} \quad
\ktableau{{3}&{\underline{3}}&{\underline{4}}\\{2}&{2}&{2}&{\underline{2}}\\{1}&{1}&{1}&{1}}\quad
\ktableau{{3}&{\underline{3}}&{\underline{3}}\\{2}&{2}&{2}&{\underline{3}}\\{1}&{1}&{1}&{2}}\quad
\ktableau{{3}&{\underline{3}}&{\underline{4}}\\{2}&{2}&{3}&{\underline{3}}\\{1}&{1}&{1}&{2}}\quad
\ktableau{{3}&{\underline{4}}&{\underline{4}}\\{2}&{2}&{3}&{\underline{3}}\\{1}&{1}&{1}&{2}}
\ .
$$
We will also need ${\rm sat}(\Psi({\mathcal D}))$, which here are
$$\ktableau{{3}&{3}&{3}\\{2}&{2}&{2}&{2}\\{1}&{1}&{1}&{1}} \quad
\ktableau{{3}&{3}&{3,4}\\{2}&{2}&{2}&{2}\\{1}&{1}&{1}&{1}}\quad
\ktableau{{3}&{3}&{3}\\{2}&{2}&{2}&{3}\\{1}&{1}&{1}&{1,2}}\quad
\ktableau{{3}&{3}&{4}\\{2}&{2}&{2,3}&{3}\\{1}&{1}&{1}&{1,2}}\quad
\ktableau{{3}&{3,4}&{4}\\{2}&{2}&{2,3}&{3}\\{1}&{1}&{1}&{1,2}} \ .
$$
%(The Kazhdan-Lusztig polynomial $P_{v,w}(q)=1+2q+q^2+q^3$, and
%$P_{v,w}(q)\preceq K_{v,w}(q)$.)
\qed

\end{example}

\begin{lemma}\label{lemma:map g}
Suppose $\mathcal{D}\in{\rm drift}(v,w)$ and $T=\Psi(\mathcal{D})$. Then:
\begin{itemize}
\item[(i)] $T$ is a semistandard Young tableau (i.e., $\Psi$ is well-defined);

\item[(ii)] $\Psi$ is an injection;

\item[(iii)] if the $j$-th column of $\lambda(w)$ has no special box, then $T(i,j)=i$ for all $1\le i\le \hat{i}$; and

\item[(iv)] ${\rm wt}(\mathcal{D})=\ex(\sat(T))=\ey(T)$.
\end{itemize}
\end{lemma}
\begin{proof}
    For (i) notice that since each corner of $\lambda(w)$ is special, it is assigned a finite number. Hence
(\ref{eqn:prescription}) assigns each box of $\lambda(w)$ a finite number. Moreover, the column semistandardness conditions are immediate from (\ref{eqn:prescription}). We now establish the row semistandardness condition $T(i,j)\le T(i,j+1)$, considering the two
cases that can occur.

\noindent
\emph{Case 1: $(i,j)$ is atop a special box.} That is, there is a special box $(i_0,j)$ and  $i=\widehat{i_0}$.
Then if $(i,j+1)\in\lambda(w)$, it is atop another special box: Suppose not. Then let the arm and leg
length of $(i,j)$ be ${\mathcal L}$. Note that since $\lambda(w)$ is a Young diagram,
$(i-{\mathcal L}+1, j+{\mathcal L}+1)\not \in \lambda(w)$. Thus there is a smallest integer $k$ such that $1\leq k\leq {\mathcal L}$
and $(i-k+1, j+k+1)\not\in \lambda(w)$. For this $k$ note that $(i-k+1,j+1)$ has equal arm and leg length
equal, no other special boxes are above it (by assumption) and no boxes to strictly to its right can be special (their leg lengths
are strictly longer than their arm lengths). Hence $(i-k+1,j+1)$ is special, but this is a contradiction.

Now that we know that both $(i,j)$ and $(i,j+1)$ are atop special boxes, hence $T(i,j)$ and $T(i,j+1)$ are the heights
of the boxes $(i,j)$ and $(i,j+1)$ in the drift configuration  $\mathcal{D}$. From this interpretation,
it is clear that $T(i,j)\le T(i,j+1)$.

\noindent
\emph{Case 2: $(i,j)$ is not atop a special box.} In this situation, by (\ref{eqn:prescription}):
$$T(i,j)\le T(i-1,j+1)+1\le T(i,j+1).$$

%\noindent
%emph{Case 2: $(i,j)$ is not atop a special box, but $(i,j+1)$ is.} In this situation, by (\ref{eqn:prescription}):
%$$T(i,j)\le T(i-1,j+1)+1\le T(i,j+1).$$
%
%\noindent
%\emph{Case 3: neither $(i,j)$ nor $(i,j+1)$ is atop a special box.} Then, again by (\ref{eqn:prescription}):
%$$T(i,j)=\min\{T(i+1,j)-1, T(i-1,j+1)+1\}$$
%and
%$$T(i,j+1)=\min\{T(i+1,j+1)-1, T(i-1,j+2)+1\}.$$
%We need to show that $T(i+1,j)\le T(i+1,j+1)$ and $T(i-1,j+1)\le T(i-1,j+2)$. Both
%inequalities follow from the inductive assumption if we do the induction on $(i,j)$ by working down columns, from right to left.

Next, (ii) is immediate
since different drift configurations will lead to different initial fillings, of the boxes $(\hat{i},j)$ where
$(i,j)$ is a special box.

Now we prove (iii). First note that
$(\hat{i},j+1), (\hat{i}-1,j+2), (\hat{i}-2,j+3),\dots, (1,j+\hat{i})$
must lie in $\lambda(w)$. Otherwise suppose $k\in\mathbb{Z}_{\geq 0}$ is the smallest integer that $(\hat{i}-k, j+k+1)$ is not in
$\lambda(w)$. Since the $j$-th column does not contain a special box, $(\hat{i},j)$ is not a corner, so $(\hat{i},j+1)$ must lie in
$\lambda(w)$, and we have $k\ge 1$. Since $k$ is the smallest integer where the failure occurs,
$(\hat{i}-k+1,j+k)$ must lie in $\lambda(w)$, and therefore $(\hat{i}-k,j+k)$ lies in $\lambda(w)$.
The conclusion that $(\hat{i}-k,j)$ is deduced is a similar manner as in ``Case 1'' of (i).
%Note that none of
%\[(\hat{i}-k,j+1), (\hat{i}-k,j+2),\cdots, (\hat{i}-k,j+k)\]
%is  a special box since in each case the length of the leg is strictly longer than the arm.
%Also we have that the arm and leg lengths of the box $(\hat{i}-k,j)$ are equal. It follows that $(\hat{i}-k,j)$ is a special box,
%contradicting to our assumption of about column $j$ having no special boxes.

Now applying (\ref{eqn:prescription}) repeatedly, we have
$$T(\hat{i},j)\le T(\hat{i}-1,j+1)+1\le T(\hat{i}-2,j+2)+2\le \cdots \le T(1,j+\hat{i}-1)+\hat{i}-1,$$
and each of the boxes being considered actually lie in $\lambda(w)$, because of what we just argued.
Since $T(1,j+\hat{i}-1)=1$ (which holds because $(1,j+\hat{i})\in\lambda(w)$ so (\ref{eqn:prescription}) is assigned
using the boundary value $T(0,j+\hat{i})=0$), we have $T(\hat{i},j)\le \hat{i}$, which forces by the fact $T$ is semistandard that
$T(i,j)=i$ for $1\le i\le \hat{i}$.

In (iv), the second equality is just the definition (\ref{eqn:depth}). Now we establish the first
equality. Consider the $j$-th column of $\lambda(w)$.

\noindent
\emph{Case 1: this column contains a special box $(i,j)$.} The column contains $\hat{i}$ boxes and so each of the numbers
$1,2,\dots,(\hat{i}+d)$ appears exactly once in this column of ${\rm sat}(T)$, by the definition of ${\rm sat}$ and $\Psi$.
Hence the
number of extra entries of ${\rm sat}(T)$
in column $j$ is equal to $(\hat{i}+d)-\hat{i}=d$, which is the same as the distance moved by the continent of $(i,j)$.

\noindent
\emph{Case 2: the column contains no special box.} By (iii), there are not any extra entries in this column.

Summing up the number of extra entries in each column $j$ of $\sat(T)$, we conclude
$\ex(\sat(T))$ is equal to ${\rm wt}(\mathcal{D})$, as desired.
\end{proof}

Therefore,
$$P_{v,w}(q)=\sum_{\mathcal{D}\in {\rm drift}(v,w)} q^{{\rm wt}(\mathcal{D})}=\sum_{\mathcal{D}\in {\rm drift}(v,w)}
q^{\ey(\Psi(\mathcal{D}))}\preceq \sum_{T\in {\rm SSYT}(\lambda(w),{\bf b})} q^{\ey(T)}=H_{v,w}(q).$$
Here the first equality holds by Theorem~\ref{thm:main}(I), the second equality is by (iv), the ``$\preceq$'' is by (ii), and
the final equality is by Theorem~\ref{thm:Kvwrule}.

It remains to prove that
$$\deg H_{v,w}(q)=\deg P_{v,w}(q).$$
Since we have already proved that $P_{v,w}(q)\preceq H_{v,w}(q)$ which implies $\deg P_{v,w}(q)\le\deg H_{v,w}(q)$, we need only to prove
that $\deg H_{v,w}(q)\le\deg P_{v,w}(q)$.
To do so, we will need the following lemma.
\begin{lemma}\label{lemma:image of g}
$T\in {\rm SSYT}(\lambda(w),{\bf b})$ is in the image of
$\Psi:{\rm drift}(v,w)\to {\rm SSYT}(\lambda(w),{\bf b})$ if and only if both of the following conditions are true:
\begin{itemize}
\item[(a)]  For any box $(i,j)$ that is not equal to $({\widehat{i'}},j)$ for a special box $(i',j)$,
(\ref{eqn:prescription}) holds under the conventions (\ref{eqn:convention1}) and (\ref{eqn:convention2}).
\item[(b)] If $(i,j)$ and $(i',j')$ are any two special boxes with $(i,j)$ weakly southwest of $(i',j')$, then
$$T(\hat{i},j)-\hat{i}\le T(\widehat{i'},j')-\widehat{i'}.$$
\end{itemize}
\end{lemma}
\begin{proof}
Let ${\mathcal D}\in {\rm drift}(v,w)$. We show that $\Psi({\mathcal D})$ satisfies (a) and (b). The condition (a) holds by the
definition of $\Psi$. The condition (b) follows since $T(\hat{i},j)-\hat{i}$ equals the distance drifted by the continent containing
$(i,j)$, $T(\widehat{i'},j')-\widehat{i'}$ equals the distance drifted by the continent containing $(i',j')$, and the continent
 associated to $(i,j)$ cannot move further northeast than the continent associated to $(i',j')$.

Conversely, we now show
that every $T\in {\rm SSYT}(\lambda(w),{\bf b})$ satisfying (a) and (b) is in the image of $\Psi$. Consider the (putative)
drift configuration ${\mathcal D}$ defined as follows. To each continent of
${\mathcal D}$ associated to a special box $(i,j)$, shift it northeast by
$T({\hat i},j)-\hat{i}$ units. We first prove that each continent fits inside $B(v,w)$: Consider the continent with special box $(i,j)$.
If part of the continent is shifted out of the boundary $B(v,w)$, then by (b) there is some northeast corner of $\lambda(w)$
(i.e., a $1\times 1$ continent) that has been pushed out of $B(v,w)$ by that part of the continent.
Hence the corresponding $T$ is not in ${\rm SSYT}(\lambda(w),{\bf b})$, a contradiction.

Now, the condition (b) guarantees
that ${\mathcal D}$ can in fact be obtained without any continents overlapping.
Hence ${\mathcal D}\in {\rm drift}(v,w)$. Finally, by (a), we have $\Psi({\mathcal D})=T$.
\end{proof}

Given any $T_0\in {\rm SSYT}(\lambda(w),{\bf b})$, suppose
\begin{equation}
\label{eqn:thingtoavoid}
\mbox{there is a box $(i,j)$ in $\lambda(w)$ which is not a northeast corner and (\ref{eqn:prescription}) does
not hold}
\end{equation}
for $T=T_0$. Furthermore let us assume $(i,j)$ is chosen such that $j$ is smallest, with ties broken by taking $i$ smallest.

A brief outline of the remainder of the proof is as follows. Starting from $T_0$, we
construct a sequence $T_1,T_2,\dots\in {\rm SSYT}(\lambda(w),{\bf b})$ with increasing depth until we arrive at a $T_k$ that fails
(\ref{eqn:thingtoavoid}). This $T_{k}$ is proved to be in the image of $\Psi$. Then we show $\mathcal{D}:=\Psi^{-1}(T_k)\in{\rm drift}(v,w)$ satisfies ${\rm wt}(\mathcal{D})\geq {\rm depth}(T_0)$. From this the result follows; see (\ref{eqn:finale}).

%$$T_0(i,j)\neq\min\{T_0(i+1,j)-1, T_0(i-1,j+1)+1\}.$$
%(By convention, $T_0(i,j)=\infty$ if ($i>0$ and $(i,j)\notin \lambda$) or if ($j>m$); $T_0(i,j)=0$ if $i=0$ and $j\leq m$. Here $m$ is
%the number of columns in $\lambda$.)

Then let
$T_1\in {\rm SSYT}(\lambda(w),{\bf b})$
be the augmentation of $T_0$ obtained by setting
\begin{equation}
\label{eqn:augment}
T_1(i,j)=\min\{T_0(i+1,j)-1, T_0(i-1,j+1)+1\}
\end{equation}
and letting all other entries in $T_1$ be the same as in $T_0$.

Now we show that $T_1 \in {\rm SSYT}(\lambda(w),{\bf b})$.
To do this, we need to check semistandardness conditions
\begin{equation}\label{eq:T1_1}T_1(i,j-1)\le T_1(i,j)\le T_1(i,j+1)\end{equation}
and
\begin{equation}\label{eq:T1_2}T_1(i-1,j)< T_1(i,j)< T_1(i+1,j).\end{equation}
We first check (\ref{eq:T1_1}). The second inequality is trivial from (\ref{eqn:augment}). For the first inequality, we have
$$\aligned
&T_0(i,j-1)\le T_0(i+1,j-1)-1\le T_0(i+1,j)-1,\\
&T_0(i,j-1)\le T_0(i-1,j)+1\le T_0(i-1,j+1)+1.\\\endaligned$$
(The second line above uses the minimality of our choice of $(i,j)$.)
Hence $$T_1(i,j-1)=T_0(i,j-1)\le \min\{T_0(i+1,j)-1,T_0(i-1,j+1)+1\}=T_1(i,j).$$
Similarly for (\ref{eq:T1_2}), the second inequality is similarly trivial from (\ref{eqn:augment}),
whereas for the first inequality, we have
$$\aligned
&T_0(i-1,j)< T_0(i,j)\le T_0(i+1,j)-1,\\
&T_0(i-1,j)\le T_0(i-1,j+1)< T_0(i-1,j+1)+1,\\\endaligned$$
and hence $$T_1(i-1,j)=T_0(i-1,j)< \min\{T_0(i+1,j)-1,T_0(i-1,j+1)+1\}=T_1(i,j).$$

Next, we claim that
$${\rm depth}(T_1)\ge {\rm depth}(T_0).$$
The difference in depth between $T_1$ and $T_0$ can only be blamed on the boxes in positions $(i,j),(i,j+1)$
and $(i+1,j)$. Without loss of generality, let us assume that each of the latter two boxes
actually lie in $\lambda(w)$ (at least one of $(i,j+1)$ or $(i+1,j)$ is in $\lambda(w)$ since $(i,j)$ is assumed to not be a
northeast corner; analyzing the resulting cases is similar and easier).
Taking this into account leads to:
$$\aligned {\rm depth}(T_1)-{\rm depth}(T_0)&=T_1(i,j)-T_0(i,j)\\
&+\min\{T_1(i,j+1)-T_1(i,j), T_1(i,j+1)-T_1(i-1,j+1)-1\}\\
&-\min\{T_0(i,j+1)-T_0(i,j), T_0(i,j+1)-T_0(i-1,j+1)-1\}\\
&+\min\{T_1(i+1,j)-T_1(i+1,j-1), T_1(i+1,j)-T_1(i,j)-1\}\\
&-\min\{T_0(i+1,j)-T_0(i+1,j-1), T_0(i+1,j)-T_0(i,j)-1\}.
\endaligned
$$
For simplicity, set
\[y:=T_r(i+1,j), \ z:=T_r(i,j+1), \ u:=T_r(i+1,j-1), \ v:=T_r(i-1,j+1)\]
for $r=0,1$. Also let
\[x:=T_0(i,j), \ x':=T_1(i,j)=\min(y-1,v+1).\]
Using $\min(a,b)=(a+b-|a-b|)/2$, we have
$$\aligned {\rm depth}(T_1)-{\rm depth}(T_0) &=x'-x+\min(z-x',z-v-1)-\min(z-x,z-v-1)\\&\quad+\min(y-x'-1,y-u)-\min(y-x-1,y-u)\\
&=x'-x+\\
&\ \ \ \ \frac{2z-x'-v-1-|x'-v-1|}{2}-\frac{2z-x-v-1-|x-v-1|}{2}\\&\quad+\frac{2y-x'-u-1-|x'-u+1|}{2}-\frac{2y-x-u-1-|x-u+1|}{2}\\
&=\frac{1}{2}\big{[}(|x-u+1|+|x-v-1|)-(|x'-u+1|+|x'-v-1|)\big{]}\\
&=\frac{1}{2}[f(x)-f(x')]
\endaligned$$
where
\[f(a):=|a-u+1|+|a-v-1|.\]
It is elementary that $f(a)$ takes the minimal value throughout (real) interval
\[[\min(v+1,u-1),\max(v+1,u-1)].\]
Notice $x'$ is in the interval: $x'\ge \min(v+1,u-1)$ since $y\ge u$. On the other hand,
$x'\le v+1\le \max(v+1,u-1)$. Since $f$ attains its minimum at $x'$ then  $f(x)-f(x')\ge 0$ and so
${\rm depth}(T_1)\ge {\rm depth}(T_0)$ as required.

Repeating this procedure while the undesirable (\ref{eqn:thingtoavoid}) still is true,
we obtain successively $T_0, T_1, T_2, T_3, \cdots$. We claim that after finite number of iterations
(\ref{eqn:thingtoavoid}) finally fails for some
$T_k$, $k\ge 0$. To see this, let the
vector $\textbf{u}(T)=(u_1,u_2,\dots,u_{|\lambda(w)|})$
measure how ``far'' is $T\in {\rm SSYT}(\lambda(w),{\bf b})$ from failing (\ref{eqn:thingtoavoid}): Order the boxes in $\lambda(w)$ from
left to right, and in each column from bottom up. For example, in Example \ref{exa:433}, the order is
$$\ktableau{{3}&{6}&{9}\\{2}&{5}&{8}&{11}\\{1}&{4}&{7}&{10}}$$
For each $1\le i\le |\lambda(w)|$, define $u_i$ to be $0$ if the $i$-th box is a northeast corner or if (\ref{eqn:prescription}) holds,
otherwise let $u_i=1$. Then ${\bf u}(T)=(0,0,\dots,0)$ means that we are in the good case that
(\ref{eqn:thingtoavoid}) fails.
We define a pure reverse lex order on $\{0,1\}^{|\lambda(w)|}$: given
${\bf u}, {\bf u'}\in \{0,1\}^{|\lambda(w)|}$,
we say that ${\bf u}>{\bf u'}$ if
%$|{\bf u}|=\sum_i u_i > |{\bf u}'|=\sum_i u_i'$,} or if $|{\bf u}|=|{\bf u'}|$ and
\[u_{|\lambda(w)|}=u'_{|\lambda(w)|}, u_{|\lambda(w)|-1}=u'_{|\lambda(w)|-1},\cdots, u_{i+1}=u'_{i+1}, u_i>u'_i,\]
for some $i$. It is straightforward to check that,
at each step $t$, we have ${\bf u}(T_t)>{\bf u}(T_{t+1})$
and hence the above procedure
must eventually terminate, say at step $k$, with ${\bf u}(T_k)=(0,0,\ldots,0)$, as desired.

Let $T=T_k$ be the output of the above procedure. Now we want to apply Lemma \ref{lemma:image of g} to conclude that $T_k(i,j)$ is in the
image of $\Psi$, by
verifying its conditions (a) and (b).

Since
(\ref{eqn:thingtoavoid}) fails, every box that is not a northeast corner has (\ref{eqn:prescription}) holding.
In particular, this includes every box described by (a) and so (a) holds.

To check (b), let ${\mathcal L}:=\hat{i}-i$ be the leg length of $(i,j)$. Since
$(i,j)$ is special, ${\mathcal L}=|{\rm arm}(i,j)|$, and moreover, we can apply the argument in the proof of Lemma~\ref{lemma:map g}(iii) to the subset of the Young diagram $\lambda(w)$ consisting of those boxes strictly above row $i$ and weakly to the right of column $j$,
and conclude that the following boxes lie in $\lambda(w)$:
$$(\hat{i},j+1), (\hat{i}-1,j+2),\dots, (\hat{i}-{\mathcal L}+1, j+{\mathcal L}).$$
In particular, the boxes
$$(\hat{i},j), (\hat{i}-1,j+1), (\hat{i}-2,j+2),\dots, (\hat{i}-{\mathcal L}, j+{\mathcal L})$$
are not the northeast corners of $\lambda(w)$, hence (\ref{eqn:prescription}) holds for them by the construction of $T=T_k$. By (\ref{eqn:prescription}), we have
\begin{equation}
\label{eqn:bycons}
T(\hat{i}-m,j+m)\geq T(\hat{i},j)-m, \quad \textrm{for $m=0,1,\dots,\mathcal{L}$.}
\end{equation}

Since $(\widehat{i'},j')$ is to the right of $(\widehat{i'}, j+(\hat{i}-\widehat{i'}))$, we have
\[T(\widehat{i'},j')\ge T(\widehat{i'},j+(\hat{i}-\widehat{i'}))=T(\widehat{i}-(\widehat{i}-\widehat{i'}),j+(\hat{i}-\widehat{i'}))\ge T(\hat{i},j)-(\hat{i}-\widehat{i'}),\]
where the last inequality holds because of (\ref{eqn:bycons}) for $m={\hat i}-\widehat{i'}$, and since the hypothesis that $(i,j)$ is weakly southwest of $(i',j')$
 implies ${\hat i}-\widehat{i'}\leq \mathcal{L}-1$. Thus,
$$ T(\hat{i},j)-\hat{i}\le T(\widehat{i'},j')-\widehat{i'}.$$
Therefore condition (b) holds.

 Concluding, there exists ${\mathcal D}\in {\rm drift}(v,w)$ such that $\Psi({\mathcal D})=T_k$ and ${\rm wt}({\mathcal D})={\rm
 depth}(T_k)$. Then
 \begin{equation}
 \label{eqn:finale}
 {\rm wt}({\mathcal D})={\rm depth}(T_k)\ge {\rm depth}(T_{k-1})\ge\cdots\ge{\rm depth}(T_0)
 \end{equation}
  and so $\deg P_{v,w}(q)\ge \deg H_{v,w}(q)$.

This completes the proof of the theorem.\qed

\section{A ball of drift configurations}
%The arguments in Section~\ref{sec:PeqQ} show that the drift configurations ${\mathcal D}\in {\rm drift}(v,w)$
%are in weight preserving bijection with edge-labeled trees ${\mathcal G}$ in $DL({\mathcal T})=EL({\mathcal T})$.
%Thus what we are about to do can be interchangeably phrased in either of these equivalent presentations.
\subsection{Construction of ${\rm KL}_{v,w}$}
In order to emphasize the combinatorial relations of drift configurations to Young tableaux, consider an
equivalent formulation of drift configurations: A {\bf semistandard (ordinary) drift tableau} $T$
bijectively associated to ${\mathcal D}$ is a filling of each continent $C$ of ${\rm Pangaea}(v,w)$ by the distance $C$
has moved from ${\rm Pangaea}(v,w)$.

Similarly, a {\bf set-valued drift tableau} is a filling of each
continent by some non-empty set
of nonnegative integers; it is {\bf semistandard} if any ordinary drift tableau it contains (in the obvious sense) is semistandard.
It is {\bf limit semistandard} if it contains at least one semistandard (ordinary) drift tableau. The {\bf empty-face drift tableau}
${\mathcal E}_{v,w}$ is the set-valued drift tableau that is the union of all semistandard ordinary ones.

Define ${\rm KL}_{v,w}$ to be the
simplicial complex whose faces are indexed by limit semistandard drift tableau
and where face containment is by reverse containment of drift tableau. In particular, the vertices are labeled by limit
semistandard tableaux $(b\not\mapsto y)$ obtained by removing precisely one entry $y$ from a set ${\mathcal E}_{v,w}(b)$ of the box
$b\in \lambda(w)$, provided $|{\mathcal E}_{v,w}(b)|>1$. (It will be convenient to also consider {\bf phantom vertices} which are those
$(b\not\mapsto y)$ where $|{\mathcal E}_{v,w}(b)|=1$; these become honest vertices after coning over ${\rm KL}_{v,w}$.)

This gives an example of a tableau complex in the sense of \cite{KMY}. See Figure~4 for an example of ${\rm KL}_{v,w}$.

\begin{figure}[h]%--means top, bottom, here.
    \centering % optional - centre figure
    \includegraphics[width=.8\columnwidth]     {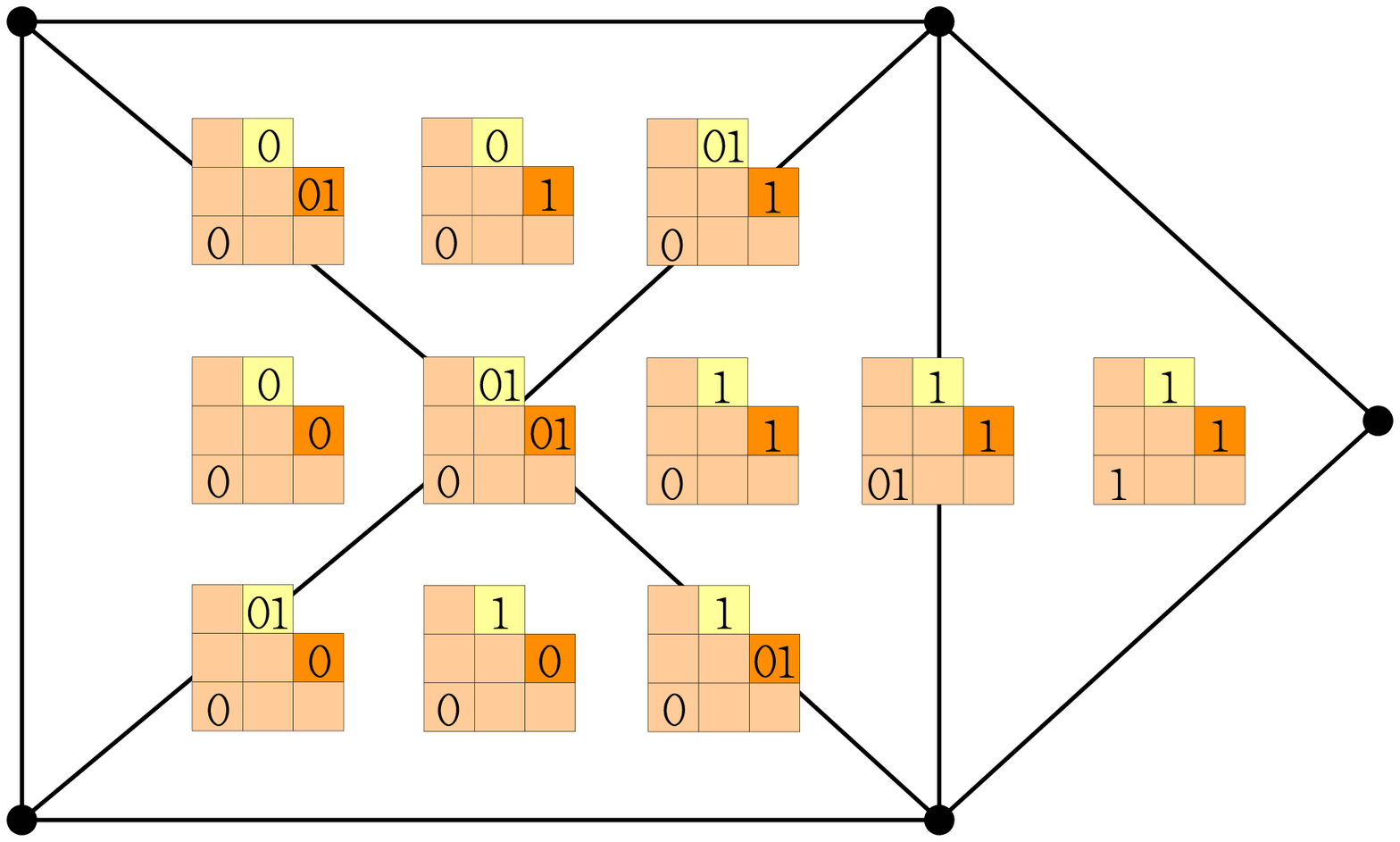}
\caption{Continuing Example~\ref{exa:433}; the interior faces of the $2$-dimensional ${\rm
KL}_{234651789\underline{10},\underline{10}954382761}$}
    \end{figure}

The claims in Theorem~\ref{thm:main} about the structure of ${\rm KL}_{v,w}$ then follow immediately from
\cite[Theorem~2.8]{KMY}. This was, we conclude that the interior faces of ${\rm KL}_{v,w}$ are labeled by
semistandard set-valued drift tableaux while the exterior faces are labeled by non-semistandard but limit semistandard tableaux.
Also the codimension of a face ${\mathcal D}$ is $|{\mathcal D}|-\#{\rm continents}$, the number of ``extra'' entries of ${\mathcal D}$.

\subsection{$K$-polynomials of ${\rm KL}_{v,w}$}
Let us take this opportunity to formalize a connection between the
$K$-polynomials of ${\rm KL}_{v,w}$ and $P_{v,w}(q)$. We will utilize facts collected
about general tableau complexes from \cite[Section~4]{KMY}.
Let $V$ be the set of vertices of a simplicial complex $\Delta$ and set $R=\Bbbk[\Delta]$ to be the polynomial
ring in variables $x_{\mathfrak v}$ for ${\mathfrak v}\in V$. This is the ambient ring for the {\bf Stanley-Reisner ideal}
$I_{\Delta}=\left\langle \prod_{{\mathfrak v}\in F} x_{\mathfrak v}: F \mbox{\ is not a face of $\Delta$}\right\rangle$
of $\Delta$, and $R/I_{\Delta}$ is the {\bf Stanley-Reisner ring}. We use the alphabet
${\bf t}_{\mathfrak v}=\{t_{\mathfrak v}:{\mathfrak v}\in V\}$ for the finely graded Hilbert series ${\rm Hilb}(R/I_\Delta; {\bf t})$
and $K$-polynomials ${\mathcal K}(R/I_{\Delta},{\bf t})$.

Let us define a family of polynomials for $v\leq w$ where $w$ is covexillary.
We will see this is a hybrid of the $K$-polynomial of ${\rm KL}_{v,w}$
and the Kazhdan-Lusztig polynomial $P_{v,w}(q)$:
\begin{equation}
\label{eqn:hybrid}
{\mathfrak P}_{v,w}(\beta; {\bf t})=\sum_{{\mathcal D}\in SVDT(v,w)} \beta^{|{\mathcal D}|-\#{\rm continents}(v,w)}\prod_{b\in
\lambda(w)}\prod_{y\in {\mathcal D}(b)}(1-t_{(b\not\mapsto y)}),
\end{equation}
where $SVDT(v,w)$ is the set of set-valued drift tableaux associated to drift configurations in ${\rm drift}(v,w)$,  $\#{\rm
continents}(v,w)$ is the number of continents in ${\rm Pangaea}(v,w)$, $|{\mathcal D}|$ is the number of entries in ${\mathcal D}$.
There are a number of interesting specializations of this polynomial. Here we do not assume $|{\mathcal E}_{v,w}(b)|>1$, i.e.,
$(b\not\mapsto y)$ might be a phantom vertex.

By the ballness/sphereness claim of ${\rm KL}_{v,w}$ from Theorem~\ref{thm:main}, together with \cite[Theorem~4.3]{KMY} it follows that
\begin{equation}\label{eqn:ballsphereclaim}
{\mathfrak P}_{v,w}(-1;{\bf t})={\mathcal K}(R/I_{{\rm KL_{v,w}}}; {\bf t})
\end{equation}
\excise{More precisely, one can cone over ${\rm KL}_{v,w}$ so that every phantom vertex becomes an actual vertex ${\mathfrak v}$ in this
auxiliary complex $\Delta$. Then the left hand side of (\ref{eqn:ballsphereclaim}) computes the $K$-polynomial of $\Delta$. However,
since coning does not change the $K$-polynomial, the asserted equality of (\ref{eqn:ballsphereclaim}) holds.}

\excise{
For any simplicial complex $\Delta$, if we introduce new elements of the ground set that the complex is defined on and cone $\Delta$ over
them, the $K$-polynomial does not change.
In the case $\Delta={\rm KL}_{v,w}$ this can be achieved by adding ``$\infty$'' to each box of $b\in\lambda(w)$ such that  $|{\mathcal
E}_{v,w}(b)|=1$, to obtain a tableau ${\hat {\mathcal E}_{v,w}}$.
Then for all such $b$, if $y\in b$, then ${\mathfrak v}=(b\not\mapsto y)$ is a vertex and $(b\not\mapsto  \infty)$ is a cone point in the
tableau complex obtained by coning
${\rm KL}_{v,w}$. Here ${\hat {\mathcal E}_{v,w}}$ is the empty face tableau.
Then the stated results of \cite{KMY} show the left hand side of (\ref{eqn:ballsphereclaim}) actually
equals the $K$-polynomial of ${\rm KL}_{v,w}$ after coning with these points.}

One can consider a vertex decomposition of any complex $\Delta$ at a vertex ${\mathfrak v}$. This is given by
$\Delta={\rm del}_{\mathfrak v}(\Delta)\cup {\rm star}_{\mathfrak v}(\Delta)$ where ${\rm del}_{\mathfrak v}(\Delta)=\{F\in
\Delta:{\mathfrak v}\not\in F\}$ is the {\bf deletion} of ${\mathfrak v}$
and ${\rm star}_{\mathfrak v}(\Delta)=\{F\in \Delta:F\cup\{{\mathfrak v}\}\in \Delta\}$ is the {\bf star} of ${\mathfrak v}$.
Automatically one has, for ${\mathfrak v}=(b\not\mapsto y)$
\begin{equation}
\label{eqn:vertexdec}
{\mathcal K}(R/I_{{\rm KL}_{v,w}}; {\bf t})=t_{(b\not\mapsto y)}{\mathcal K}(R/I_{{\rm del}_{(b\not\mapsto y)}({\rm KL}_{v,w})}; {\bf t})
+(1-t_{(b\not\mapsto y)}){\mathcal K}(R/I_{{\rm star}_{(b\not\mapsto y)}({\rm KL}_{v,w})}; {\bf t}).
\end{equation}
By tracing the specializations below, one should eventually
interpret recursions from \cite{LS:KL} for $P_{v,w}(q)$ using (\ref{eqn:vertexdec})
and thus vertex decompositions of ${\rm KL}_{v,w}$. We do not pursue this here.

Consider
\begin{equation}
{\mathfrak P}_{v,w}(-1; t_{(b\not\mapsto y)}\mapsto 1-x_y)=\sum_{{\mathcal D}\in SVDT(v,w)} (-1)^{|{\mathcal D}|-\#{\rm continents}(v,w)}
{\bf x}^{\mathcal D},
\end{equation}
where
\[{\bf x}^{\mathcal D}=\prod_{i\geq 0}x_i^{\tiny{\#\mbox{$i$'s appearing in ${\mathcal D}$}}}.\]

Another specialization is given by
\begin{equation}
\label{eqn:likeschur}
{\mathfrak P}_{v,w}(0; t_{(b\not\mapsto y)}\mapsto 1-x_y)=\sum_{{\mathcal D}\in SSDT(v,w)}
{\bf x}^{\mathcal D},
\end{equation}
where $SSDT(v,w)$ is the set of ordinary, semistandard drift tableau associated to $v,w$. (In setting $\beta=0$ we take the convention
that $0^0=1$ in (\ref{eqn:hybrid}).)

Finally, by considering the principal specialization of (\ref{eqn:likeschur}) we have
\[{\mathfrak P}_{v,w}(0; t_{(b\not\mapsto y)}\mapsto 1-q^y)=P_{v,w}(q).\]

%%%%%%%%%%%%%%%%%%%%%%%%%%%%%%%%%%%%%%%%%%%%%%%%%%%%%%%%%%%%%%%%%%%%%%%%

\section{The proof of Theorem \ref{thm:main}(I)}

\subsection{Proof of $Q_{v,w}(q)=P_{v,w}(q)$}
\label{sec:PeqQ}
We give a weight-preserving bijection between ${\rm drift}(v,w)$ and the trees weight-enumerated by
Lascoux's rule \cite{Lascoux} for
$P_{v,w}(q)$. We mostly follow the presentation of his rule found in \cite[6.3.29]{Billey.Lakshmibai}.

Given $\mathcal{D}\in {\rm drift}(v,w)$, construct a rooted, edge-labeled tree ${\mathcal T}$ as follows.
Associate to each continent $C$ a non-root vertex of ${\mathcal T}$.
Moreover if the special box $b$ of $C$ is southwest of the special box $b'$ of an adjacent continent $C'$,
then we draw an edge between the corresponding vertices.
If there is no special box strictly southwest of $b$, then the corresponding vertex
is joined to the root of ${\mathcal T}$.

Thus, each $1\times 1$ continent $C=\{(h,\lambda(w)_h)\}$ (equivalently, those that come from northeast corners of $\lambda(w)$)
corresponds to a leaf $p$ of ${\mathcal T}$.
Now we bound the edge incident  to $p$ by $b_h-h$, where
\[b_h=\max \{m \ | \ B(v,w)_{m}\geq \lambda(w)_h+m-h\}.\]
Let $DL({\mathcal T})$ be the
set of all edge labelings of ${\mathcal T}$ by nonnegative integers such that the labels weakly increase from root
to leaf. For any edge labeled tree ${\mathcal G}$ let $|{\mathcal G}|$ be the sum of the edge labels of ${\mathcal G}$.

For example, below are the trees for drift configurations in Figure \ref{figure3}. The framed number below each leaf is the bound for that
leaf.
  \begin{figure}[h]%--means top, bottom, here.
    \centering % optional - centre figure
    \includegraphics[width=1\columnwidth]{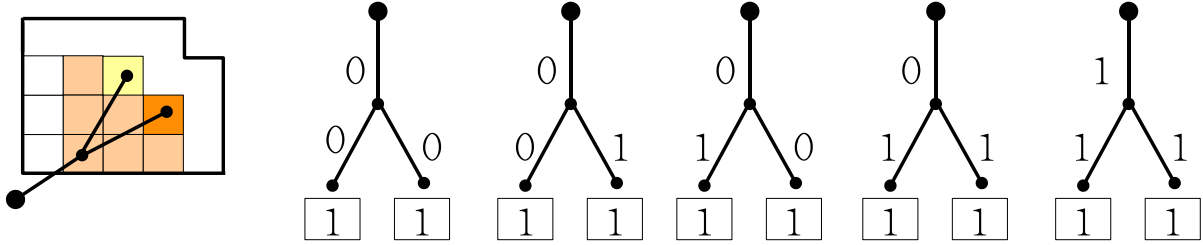}
\caption{\label{figure5} Edge labeled trees from $DL({\mathcal T})=EL({\mathcal T})$,
respectively corresponding to drift configurations in Figure \ref{figure3}.}
    \end{figure}

\begin{lemma}
There is a bijection $\Phi:{\rm drift}(v,w) \to DL({\mathcal T})$ such that ${\rm wt}({\mathcal D})=|\Phi({\mathcal D})|$.
\end{lemma}
\begin{proof}
Define $\Phi({\mathcal D})$ to be the edge labeling of ${\mathcal T}$ such that the edge associated to a continent $C$
(i.e., the edge whose child end is the vertex associated to $C$)
is labeled by the distance that $C$ has drifted in ${\mathcal D}$. That the labels are weakly increasing
in $\Phi({\mathcal D})$ is implied by the condition that the continents do not overlap in ${\mathcal D}$.
Note that if $C$ is a $1\times 1$ continent then $b_h -h$ is the largest distance that $C$ can drift inside $B(v,w)$; this accounts for
the leaf bound (see Figure~5 for a diagram).
It is then easy
to check that $\Phi$ is the desired bijection.
\end{proof}

Lascoux's rule constructs a tree ${\mathcal T}'$ as follows:
For the partition $\lambda(w)$, the {\bf parenthesis-word} is a word using ``('' and ``)'' and
obtained by walking with east and south steps along the northeast
border of $\lambda(w)$. We record a ``('' for each east step and a ``)'' for each south step.
Now pair left and right parentheses starting from the
the closest pairs ``$()$''. Each pair corresponds to a vertex of the tree, the closest pairs
are associated to leaves and a pair encloses its children.
Unpaired parentheses do not contribute to the tree. This process results in a directed forest.
Finally, we introduce an additional root and attach an edge to the root of each tree in the forest.

\begin{lemma}
There is a graph isomorphism $\delta:{\mathcal T}\to {\mathcal T}'$. Moreover
under this isomorphism if $v$ corresponds to a $1\times 1$ continent associated to a corner $c$ of $\lambda(w)$,
then $\delta(v)$ corresponds to a closest parenthesis pair associated to the same corner $c$.
\end{lemma}
\begin{proof}
Each leaf of ${\mathcal T}$ corresponds to a corner $c$ of $\lambda(w)$. On the other hand, this corner
gives rise to a closest pair ``$(\ )$'' in Lascoux's construction, which corresponds to a leaf of ${\mathcal T}'$.
Thus we can construct a bijection between the leaves of the two trees, which we now argue extends
to the bijection $\delta$ between the two trees themselves.

A continent $C$ is a {\bf $z$-continent} if it is defined by a $z$-special box $b$. Fix a vertex $v\in {\mathcal T}$ associated
to such a continent. By construction, each child of $v$ is a vertex $\{v'\}$ associated to a $y$-continent $C'$ adjacent
and northeast of $C$ in ${\rm Pangaea}(v,w)$, where $y<z$. Since $b\in C$ is a special box, by
using the fact that $|{\rm arm}(b)|=|{\rm leg}(b)|$ we have that the column $b$ is in corresponds to a $($ and the
row $b$ in in corresponds to a $)$ where these two parentheses are paired with one another in the parenthesis word. Clearly, this pair
gives a vertex $v'\in {\mathcal T'}$, and all vertices of ${\mathcal T'}$ arise this way. That is, there is a bijection at the level of
vertices $\delta:{\mathcal T}\to {\mathcal T}'$. Moreover, that the children of $\delta(v)$ are exactly $\{\delta(v')\}$ (for children
$v'$ of $v$) is also immediate from the constructions of ${\mathcal T}$ and ${\mathcal T}'$
\excise{We will build up the isomorphism by induction on $z\geq 0$. Let a $z$-continent refer to a continent defined  by a $z$-special
box.
In the base case $z=0$ continents are $1\times 1$ northeast corners of $\lambda(w)$. These correspond to leaves of
${\mathcal T}$. On the other hand, these corners are associated to closest parentheses, which give rise to leaves of ${\mathcal T}'$.
Hence the second sentence of the lemma's statement will hold for our isomorphism.

Now suppose $z>0$, and consider a vertex $v\in {\mathcal T}$ associated to a $z$-special box $b$

consider a vertex $v'$ of ${\mathcal T}'$. In the parenthesis word this corresponds to
\begin{equation}
\label{eqn:balanced}
\cdots (A_1\ A_2 \ A_3 \cdots \ A_t )\cdots
\end{equation}
where the ``$($'' and ``$)$'' correspond to $v'$ and the $A_i$'s are balanced parenthesis words and correspond to children of $v'$.
Moreover, we inductively assume
the outermost pair of parentheses of each $A_i$ biject with vertices of ${\mathcal T}$ associated to $(z-1)$-continents. Now look at the
box $b$ that is in the column associated to the ``$($'' and in the row associated to the
``$)$''. The balanced condition implies $|{\rm arm}(b)|=|{\rm leg}(b)|$. Finally, by the form of (\ref{eqn:balanced})
no special boxes can appear among $\{b\}\cup {\rm arm}(b)\cup {\rm leg}(b)$. Thus it follows $b$ is $z$-special and corresponds to
a vertex $v\in {\mathcal T}$ whose children biject with the children of $v'\in {\mathcal T}'$.}
\end{proof}
%Each continent $C$ corresponds to a pair of parentheses: let $x$ be the special box of $C$. On the northeast boundary of $\lambda(w)$,
%the horizontal edge above $x$ gives ``('', and the vertical edge on the right hand side of $x$ gives ``)''. It can be inductively
%verified that this indeed gives a pair in the parenthesis-word. It is easy to see that, if the special box $x$ of a continent $C$ is in
%the southwest direction of the special box $x'$ of a continent $C'$, then the corresponding pairs of parentheses satisfy the containing
%relation.

Lascoux's rule similarly defines increasing edge labelings $EL({\mathcal T})$ on ${\mathcal T}$ as we did for $DL({\mathcal T})$.
It remains to check that these labelings are the same as the ones in $DL({\mathcal T})$.
For this, we only need to show that the bound attached to the leaves are the same.
In \cite[6.3.29, Step 2]{Billey.Lakshmibai}, for each given leaf, a
bigrassmannian permutation is determined in three sub-steps, from which Lascoux's leaf bounds are determined.
We now explain these steps. (For readers comparing what follows with \cite{Billey.Lakshmibai}, note
their $x$ is our ${\widetilde w}=w^{-1}w_0$ while their $w$ is our ${\widetilde v}=v^{-1}w_0$.)

The reader may find the following diagram useful for the description of Lascoux's labeling process:

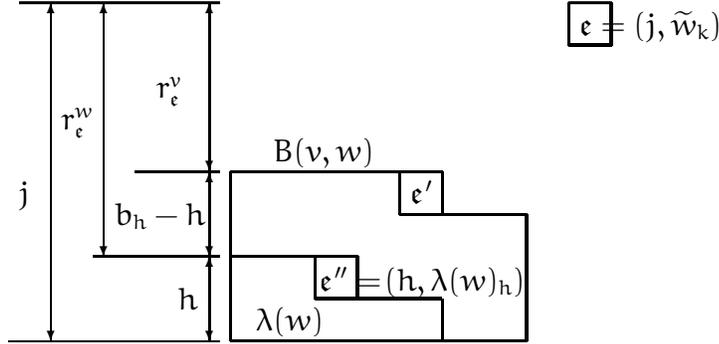
\begin{figure}[h]\setlength{\unitlength}{.8pt}
\begin{picture}(200,150)
\thicklines
\put(0,0){\line(1,0){140}} \put(140,0){\line(0,1){60}}
\put(140,60){\line(-1,0){40}} \put(100,60){\line(0,1){20}}
\put(100,80){\line(-1,0){100}} \put(0,0){\line(0,1){80}}
\put(40,20){\line(1,0){60}}
\put(0,40){\line(1,0){60}}
%\put(20,0){\line(0,1){40}}
\put(40,20){\line(0,1){20}}
\put(60,20){\line(0,1){20}}
%\put(80,0){\line(0,1){20}}
\put(100,0){\line(0,1){20}}
% The northeast box
\put(160,140){\line(0,1){20}}
\put(160,140){\line(1,0){20}}
\put(180,140){\line(0,1){20}}
\put(160,160){\line(1,0){20}}

\put(80,60){\line(1,0){20}}
\put(80,60){\line(0,1){20}}
\put(85,65){${\mathfrak e}'$}
\thinlines
\put(12,5){$\lambda{(w)}$}
\put(20,85){$B(v,w)$}
\put(165,145){${\mathfrak e}=(j,{\widetilde w}_k)$}
\put(42,24){${\mathfrak e}''=\!(h,\lambda(w)_h)$}

\put(-5,0){\line(-1,0){100}}
\put(-25,15){$h$}
\put(-10,20){\vector(0,1){20}}\put(-10,20){\vector(0,-1){20}}
\put(-65,40){\line(1,0){60}}
\put(-55,55){$b_h-h$}
\put(-10,60){\vector(0,1){20}}\put(-10,60){\vector(0,-1){20}}
\put(-5,80){\line(-1,0){40}}
\put(-35,115){$r^v_{{\mathfrak e}}$}
\put(-10,120){\vector(0,1){40}}\put(-10,120){\vector(0,-1){40}}
\put(-100,160){\line(1,0){95}}
\put(-80,100){$r^w_{{\mathfrak e}}$}
\put(-60,100){\vector(0,1){60}}\put(-60,100){\vector(0,-1){60}}
\put(-100,65){$j$}
\put(-85,100){\vector(0,1){60}}\put(-85,100){\vector(0,-1){100}}
\end{picture}
\caption{Diagram for the proof of $P_{v,w}(q)=Q_{v,w}(q)$}
\end{figure}

Sub-step (1) [leaves $p$ of ${\mathcal T}$ correspond to distinct numbers in
the code of  $\widetilde{w}$]: The code $(c_1,\dots,c_n)$ of $\widetilde{w}$ is given by
\[c_i=\#\{j>i \ | \ {\widetilde w}_j<{\widetilde w}_i\}=\#\{\mbox{boxes of $D(w)$ in row $i$}\}.\]
Recall $\lambda(w)$ is the result of sorting this code into decreasing order.
A leaf $p$ of ${\mathcal T}$ corresponds to a corner ${\mathfrak e}''=(h,\lambda(w)_h)$ of $\lambda(w)$.
Associate $\lambda(w)_h$ to $p$. This $\lambda(w)_h$ is equal to $c_i$ for some $i$. Clearly a different $c_i$ is assigned
to each~$p$.

Sub-step (2) [$\lambda(w)_h$ gives a crossing of $\widetilde{w}$]: By definition, a {\bf crossing} of $\widetilde{w}$ is a
4-tuple $(i,j,j+1,k)$ satisfying
\begin{equation}
\label{eqn:crossing}
\widetilde{w}_{j+1}\le \widetilde{w}_k<\widetilde{w}_i\le \widetilde{w}_j,\quad \widetilde{w}_i=\widetilde{w}_k+1
\mbox{\quad for $i\le j<k$;}
\end{equation}
cf. \cite{treillis}.
Now given the ${\mathfrak e}''$ associated to $p$, there is
a unique essential box ${\mathfrak e}$ in $D(w)$
that is diagonally northeast of ${\mathfrak e}''$. We define $j$ and $k$ by declaring that the
coordinates of ${\mathfrak e}$ are $(j, {\widetilde w}_k)$. Let $i$ be such that
$\widetilde{w}_i=\widetilde{w}_k+1$.

We claim that $(i,j,j+1,k)$ forms a crossing. Let us first check the weak inequalities of
$\widetilde{w}_{j+1}\le \widetilde{w}_k<\widetilde{w}_i\le \widetilde{w}_j$ (the strict inequality being
true by definition). For the rightmost inequality, we have
${\widetilde w}_j=w^{-1}w_0(j)=w^{-1}_{n-j+1}$, which in words is the column position of the $\bullet$ of $G(w)$
that necessarily must be to the right of ${\mathfrak e}$, which itself is in column ${\widetilde w}_k$. In other
words ${\widetilde w}_k\leq {\widetilde w}_j$. Now, for the leftmost inequality, note ${\widetilde w}_{j+1}=w^{-1}w_0(j+1)=w^{-1}(n-j)$
which is the column position of the $\bullet$ of $G(w)$ in row $j+1$. Since ${\mathfrak e}$ is an essential box, that $\bullet$
must be weakly to the left, i.e., ${\widetilde w}_{j+1}\leq {\widetilde w}_{k}$, as desired. It remains to check
$i\leq j$ and $j<k$. For the former inequality, we compute $w{\widetilde w}_i=n-i+1$ which is the row position of the $\bullet$
of $G(w)$ in column ${\widetilde w}_i$. Since ${\mathfrak e}$ is an essential box, the $\bullet$ is weakly below the ${\mathfrak e}$, i.e., $i\leq j$. Similarly, for the latter inequality, we consider $w{\widetilde w}_k=n-k+1$, which is the position of the $\bullet$
of $G(w)$ in column ${\widetilde w}_k$. This must be strictly above the ${\mathfrak e}$, i.e., $j<k$.

Now associate
the crossing $(i,j,j+1,k)$ to $p$ (and hence $\lambda(w)_h$).
Actually, the description in \cite{Billey.Lakshmibai} gives a different way to assign a crossing
to $p$. However, it is straightforward to check that their crossing is same as the one described above.

%(The reason that $(i,j,j+1,k)$ is called a crossing
%because in the picture, the line segments $\overline{A_1A_2}$ and
%\overline{B_1 B_2}$ intersect.)
%\begin{figure}[h]\setlength{\unitlength}{.6pt}
%\begin{picture}(200,160)
%The frame
%  \put(0,0){\line(0,1){160}}
%  \put(0,160){\line(1,0){160}}
%  \put(160,160){\line(0,-1){160}}
%  \put(160,0){\line(-1,0){160}}
%first draw e'
%  \put(20,40){\line(1,0){20}}
%  \put(20,20){\line(1,0){20}}
%  \put(20,40){\line(0,-1){20}}
%  \put(40,40){\line(0,-1){20}}
%\put(23,25){$e$}
%%first draw e'
%  \put(60,80){\line(1,0){20}}
%  \put(60,60){\line(1,0){20}}
%  \put(60,80){\line(0,-1){20}}
%  \put(80,80){\line(0,-1){20}}
%\put(63,65){$e'$}
%%the point A1
%  \put(90,50){{\circle*{5}}}
%  \put(185,45){$A_1=(i,w_i)$}\put(180,50){\vector(-1,0){80}}
%%the point B1
%  \put(130,70){{\circle*{5}}}
%  \put(185,65){$B_1=(j,w_j)$}\put(180,70){\vector(-1,0){40}}
%%the point A2
%  \put(50,90){{\circle*{5}}}
%  \put(-180,80){$B_2=(j+1,w_{j+1})$}\put(-20,90){\vector(1,0){50}}
%%the point B2
%  \put(70,150){{\circle*{5}}}
%  \put(185,147){$A_2=(k,w_k)$}\put(180,150){\vector(-1,0){90}}
%connect A1A2
%  \put(90,50){\line(-1,5){20}}
%connect B1B2
%  \put(130,70){\line(-4,1){80}}
%\end{picture}
%\end{figure}

Sub-step (3) [each crossing gives a maximal bigrassmannian
$[a,b,c,d]$ below $\widetilde{w}$]:
%We will describe $[a,b,c,d]$ momentarily, but assuming it for now,
%the leaf bound is
%$${\rm distance}([a,b,c,d],\widetilde{v})=\max\{r\ge 0| [a-r,b+r,c+r,d-r]\le \widetilde{v}\, \}. $$
Here $[a,b,c,d]$ denotes
\[(1,\dots, a, a+c+1,\dots, a+c+b, a+1,\dots, a+c, a+c+b+1,\dots, a+b+c+d)\in S_n.\]

Lascoux's rule corresponds $(i,j,j+1,k)$ to a maximal bigrassmannian
\[[z,j-z, {\widetilde w}_k-z, n-{\widetilde w}_k-j+z],\]
where
\[z=\#\{p<j: {\widetilde w}_p<{\widetilde w}_k\}.\]
Notice $z$ is the number
of $\bullet$'s in $G(w)$ weakly southwest of ${\mathfrak e}=(j,{\widetilde w}_k)$, i.e.
\begin{equation}
\label{eqn:simplefact}
z=r^w_{\mathfrak{e}}.
\end{equation}

This concludes Sub-step (3) of step 2 of \cite{Billey.Lakshmibai}.

Lascoux's rule then assigns to $p$ the following leaf bound:
\[{\rm distance}([z, j-z, {\widetilde w}_k-z, n-{\widetilde w}_k-j+z],
\widetilde{v}),\]
where
\[{\rm distance}([a,b,c,d],\widetilde{v})=\max\{r\ge 0| [a-r,b+r,c+r,d-r]\leq \widetilde{v}\, \},\]
and where ``$\leq$'' refers to Bruhat order on $S_n$.

This completes the description of Lascoux's algorithm.

Recall $r^v_{(a+b, a+c)}$ equals the number of dots of $G(v)$ weakly southwest of
$(a+b,a+c)$. Observe the following fact, whose proof is straightforward to argue (and also follows from the deeper
developments in \cite{treillis}):
\begin{lemma}\label{lemma:compare bigrass}
 For any bigrassmannian permutation $[a,b,c,d]$ and permutation $\widetilde{v}$ in
 $S_n$,
 the inequality $[a,b,c,d]\leq \widetilde{v}$ is equivalent to
$r^v_{(a+b, a+c)}\le a$, where $v=w_0\widetilde{v}^{-1}$.\qed
\end{lemma}

\begin{proposition}
\label{prop:finalcalc}
The leaf bounds on $DL({\mathcal T})$ and $EL({\mathcal T})$ are the same.
\end{proposition}
\begin{proof}
By Lemma \ref{lemma:compare bigrass},
\begin{equation}\label{eq:distance}
\aligned & [z-r,j-z+r, {\widetilde w}_k-z+r, n-{\widetilde w}_k-j+z-r]\leq \widetilde{v}\\
\iff & r^v_{(z-r)+(j-z+r),
(z-r)+({\widetilde w}_k-z+r)}\leq z-r\\
 \iff & r^v_{(j,{\widetilde w}_k)}\leq z-r\\
\iff & r^v_{{\mathfrak e}}\leq z-r.\endaligned
\end{equation}

Hence, the maximal $r$ such that any of the inequalities
(\ref{eq:distance}) hold is
\[r=z-r^v_{{\mathfrak e}}=r^w_{{\mathfrak e}}-r^v_{{\mathfrak e}},\]
where we have used (\ref{eqn:simplefact}).

In terms of drift configurations, $r$ is the largest distance that a corner
${\mathfrak e}''=(h,\lambda(w)_h)$ can be moved diagonally northeast and remain in $B(v,w)$
(cf. \cite[Lemma 5.7]{Li.Yong}).
%and this is equivalent to the
%condition that the continents are inside $B(v,w)$, because the
%maximal distance that $e$ can be moved in the northeast direction is
%$b_h-h$.
By the definition of $B(v,w)$, $b_h=j-r^v_{\mathfrak{e}}$. It is also easy to
check that $j=h+r^w_{\mathfrak{e}}$  (again by \cite[Lemma 5.7]{Li.Yong}). Then
\[b_h-h=j-r^v_{\mathfrak{e}}-h=(j-h)-r^v_{\mathfrak{e}}=r^w_{\mathfrak{e}}-r^v_{\mathfrak{e}}=r.\]
This completes the proof of the proposition.
\end{proof}

%(iii) Finally, for each continental drift $D$, we construct a labeling of the tree as below.  For an edge of the tree, let $p$ be the
%lower end of the edge (suppose the tree is drawn with the root at top), and let $c$ be the corresponding special box in $\lambda(w)$,
%associated to a continent $C$. Then we label the edge by the distance that $C$ moves from the original pangaea.
%We show that this gives a one-one correspondence between continental drifts $\mathcal{D}$ and labelings $T$ of the tree. Obviously, the
%total movement $|\mathcal{D}|$ is equal to the sum of labels $|T|$. Different continental drifts give different labelings. Giving a
%labeling, we can recover a continental drift $D$ by move each continent $C$ in the northeast direction $\ell$ steps, where $\ell$ is the
%label on the edge whose lower-end vertex is the one corresponding to $C$.
By Lascoux's rule,
\[P_{w_0 {\widetilde v}, w_0 {\widetilde w}}(q)(=P_{w_0v^{-1}w_0, w_0w^{-1}w_0}(q)=P_{v,w}(q))=\sum q^{|T|}\]
where the sum is over $EL(T)$ and $|T|$ is the total sum of the edge labels.
Since we have established the desired weight-preserving bijection, the claim $Q_{v,q}(q)=P_{v,w}(q)$ then follows.

\begin{remark}
\label{remark:twosymmetries}
There are two basic symmetries of the Kazhdan-Lusztig polynomials: (1) $P_{v,w}(q)=P_{w_0 v^{-1} w_0,w_0 w^{-1} w_0}(q)$ and
(2) $P_{v,w}(q)=P_{v^{-1},w^{-1}}(q)$. The symmetry (1) is manifest in our rule and ${\rm drift}(w_0 v w_0, w_0 w w_0)$
is obtained by transposing the drift configurations of ${\rm drift}(v,w)$. For (2), it is an exercise to prove
that $\lambda(w)=\lambda(w^{-1})$ and $B(v,w)=B(v^{-1},w^{-1})$ and so ${\rm drift}(v^{-1},w^{-1})={\rm drift}(v,w)$.
\end{remark}

\begin{remark}
\label{remark:anotherbound}
From Theorem~\ref{thm:main}(I) it is not hard to show the following. For $w, v\in S_n$ where $w$ is covexillary  and $v\le
w$, let
$k$ be the number of special boxes of $\lambda(w)$ and let $m=\lfloor\frac{n-k+1}{2}\rfloor$. If $[m]_q=1+q+\cdots +q^{m-1}$, then
$[q^{i}]P_{v,w}(q)\le [q^{i}]([m]_q)^k$ for all $i$. In particular, $P_{v,w}(1)\le m^k$.
\end{remark}

\section{Another $q$-analogue of multiplicity}

We can think of $H_{v,w}(q)$ as a $q$-analogue of Hilbert-Samuel multiplicity, in the sense that $H_{v,w}(1)={\rm mult}_{e_v}(X_w)$.
Let us point out that in the covexillary setting, there is another $q$-analogue available. As in Theorem~\ref{thm:main}(II),
regard each box of $\lambda(w)$ as a separate country; the
``drift configurations'' are precisely the \emph{pipe dreams} $P\in {\rm Pipes}(v,w)$ in \cite{Li.Yong}.
Now let \[{\widetilde {\rm wt}}(P)=q^d\] where $d$ is the total of the distance drifted by the countries, and set
\[{\widetilde H}_{v,w}(q)=\sum_{\rm P\in {\rm Pipes}(v,w)} {\widetilde {\rm wt}}(P).\]
In the following theorem we use the standard $q$-notation:
\[[a]_q=1+q+\cdots+q^{a-1}
\mbox{\ and ${a\choose b}_q=\frac{[a]_q[a-1]_q\cdots[a-b+1]_q}{[b]_q\cdots [1]_q}$.}\]

\begin{theorem}
\[{\widetilde H}_{v,w}(q)=q^{-\sum_{i\geq 1}(i-1)\lambda_i}\det\left({b_i+\lambda_i-i+j-1\choose \lambda_i-i+j}_q\right)_{1\leq i,j\leq
\ell(\lambda)},\]
where $\ell(\lambda)$ is the number of nonzero parts of $\lambda$ and ${\bf b}={\bf b}(\Theta_{v,w})$.
\end{theorem}
\begin{proof}
For brevity, we refer the reader to the setup of \cite[Sections~5.2 and~6.2]{Li.Yong}.
Notice that
\[s_{\lambda,{\bf b}}(1,q,q^2,q^3,\ldots)=\det\left({b_i+\lambda_i-i+j-1\choose \lambda_i-i+j}_q\right)_{1\leq i,j\leq \ell(\lambda)}\]
where the lefthand side of the equality is the principal specialization
of the (single) flagged Schur
polynomial for shape $\lambda(w)$ with flag ${\bf b}={\bf b}(\Theta_{v,w})$.

Given a pipe dream $P\in {\rm Pipes}(v,w)$ that corresponds to a flagged semistandard Young
tableau $T$, write
\[{\rm wt}_x(P):={\rm wt}_x(T)\]
to mean the usual multivariate weight assigned
to $T$ (i.e., so that $s_{\lambda,{\bf b}}(x_1,x_2,x_3,\ldots)=\sum_{T} {\rm wt}_x(T)$).
Let ${\rm wt}_q'(P)$ be the principal specialization of ${\rm wt}_x(P)$ given by
$x_i\mapsto q^{i-1}$
and finally set
\[{\rm wt}_q(P)=q^{-\sum_{i\geq 1} (i-1)\lambda_i}\times {\rm wt}_{q}'(P).\]

It remains to show that for each $P$, ${\rm wt}_q(P)={\widetilde {\rm wt}}(P)$. To do this,
let us induct on ${\widetilde {\rm wt}}(P)\geq 0$. The base case that ${\widetilde {\rm wt}}(P)=0$, i.e.,
where $P$ is the starting configuration holds since ${\rm wt}'_q(P)=q^{\sum_{i\geq 1} (i-1)\lambda_i}$.

Now suppose
${\widetilde {\rm wt}}(P)>0$. Then there is a $P'$ such that a move of the form
\[\begin{matrix}
\cdot & \cdot\\
+ & \cdot
\end{matrix}\ \ \mapsto \ \
\begin{matrix}
\cdot & +\\
\cdot & \cdot
\end{matrix}\]
in some $2\times 2$ subsquare of $[n]\times [n]$
brought us to $P$ (and no other $+$ in $P'$ has changed). Thus, we can compare ${\rm wt}_x(P')$
and ${\rm wt}_x(P)$: the latter only differs from the former in that some factor of $x_i$ changed to
$x_{i+1}$ (where $i$ and $i+1$ are the rows changed by the move above). Hence
applying induction we have
\[{\rm wt}_q(P)={\rm wt}_q(P')\times q={\widetilde {\rm wt}}(P')\times q={\widetilde {\rm wt}}(P),\]
as desired.
\end{proof}

It is clear from Theorem~\ref{thm:main}
that
\[P_{v,w}(q)\preceq {\widetilde H}_{v,w}(q).\]
With the same proof that we used for $H_{v,w}(q)$,
one shows that ${\widetilde H}_{v,w}(q)$ is upper semicontinuous.
However, in general ${\widetilde H}_{v,w}(q)\neq H_{v,w}(q)$. Moreover,
we do not know any algebraic/geometric measure for general Schubert varieties
that specializes to ${\widetilde H}_{v,w}(q)$.

\section{Concluding remarks}

We are presently unaware of any geometric proof of the inequality of Theorem~\ref{thm:mainineq}.
For general $Y$, let us assume, for simplicity of our discussion, that all odd local intersection cohomology groups vanish, and set
\[P_{p,Y}(q)=\sum_{i\geq 0}\dim({\mathcal H}_p^{2i}(Y))q^i.\]

\begin{question} Under what assumptions is either the inequality
$P_{p,Y}(q)\preceq H_{p,Y}(q)$ and/or the weaker inequality $P_{p,Y}(1)\leq H_{p,Y}(1)(={\rm mult}_{p}(Y))$ true?
\end{question}

\excise{The case $Y$ is toric variety appears to us to be related but not subsumed by
work on $h$-vectors of rational convex polytopes being \emph{equal} to \emph{global} intersection cohomology
betti numbers of their associated toric variety \cite{Stanley:hvec}.}

\excise{A simpler case is when
$Y$ is a curve, as all $Y$ admit small resolutions. Now
consider the following example: suppose
$Y$ has a node $p$ and its preimage contains $2$ points, then
$H_{p,Y}(q)=1+q$ but $P_{p,Y}(q)=2$.
Hence we should require the variety $Y$ to be at least normal to hope that
the inequality $P_{p,Y}(q)\le H_{p,Y}(q)$ holds (Schubert varieties are normal).
On the other hand, even without this restriction, $P_{p,Y}(1)\le H_{p,Y}(1)$ does
hold: Embed $Y$ into a projective space and consider a general hyperplane $H$ that is
sufficiently close
to the point $p$. Then number $m$ of intersection points
of $Y\cap H=\{y_1,\dots,y_m\}$ is the multiplicity $H_{p,Y}(1)$. Let the pull back of $\{y_1,\dots,y_m\}$ to $Z$ be $z_1,\dots,z_m$.
Denote the inverse points of $p$ by $x_1,\cdots,x_\ell$. Then each $x_i$ is close to one point
$y_{f(i)}$, and if for $i\neq i'$, $y_{f(i)}$ and $y_{f(i')}$ are different.
So we have an injection $i\mapsto f(i)$, hence $\ell\le m$,
i.e. $P_{p,Y}(1)\le H_{p,Y}(1)$.}

Our results on $H_{v,w}(q)$ are based on the degeneration, flat over ${\rm Spec}({\mathbb Z})$, given in \cite{Li.Yong}. Hence
Theorem~\ref{thm:Kvwrule} is valid over a field ${\Bbbk}$ of arbitrary characteristic and Conjecture~\ref{conj:mainintro} seems similarly
valid. However, the arguments of \cite{Li.Yong} also prove that
the projectivized tangent cones
of the Kazhdan-Lusztig varieties ${\mathcal N}_{v,w}$ are isomorphic
to those for ${\mathcal N}_{id,\Theta_{v,w}}$. It is then not hard to construct
some cograssmannian $v',w'$ with the same property.
We do not know if ${\mathcal N}_{v,w}$ and any such ${\mathcal N}_{v',w'}$ are actually isomorphic, although a number of useful
implications would be a consequence of this fact.

A number of formulae have been obtained for $P_{v,w}(q)$. For example, general, non-positive
formulae have been obtained by \cite{Billera} and \cite{Brenti}. Beyond the covexillary case, few positive formulae are known, see, e.g.,
\cite{Billey.Warrington} (which treats the $321$-hexagon avoiding case) and the references therein.
It would be interesting to try to extend our main theorems to these other contexts as well.

Finally, we believe many of the ideas of this paper can be extended to other Lie groups. In particular, we expect Theorems~\ref{thm:mainineq},
\ref{thm:main} and~\ref{thm:Kvwrule} to have analogues for (co)minuscule $G/P$, cf.~\cite{Boe}.
However, this requires sufficient technicalities that it is better left to a separate treatment.

\section*{Acknowledgements}
We thank Sara Billey, Xuhua He, Hiroshi Naruse and Alexander Woo for useful suggestions and questions that inspired this work. We also
thank Jonah Blasiak, Allen Knutson, Venkatramani Lakshmibai, Ezra Miller, Greg Warrington and the anonymous referee for helpful comments.
AY is partially supported by NSF grants DMS-0601010 and DMS-0901331.


\begin{thebibliography}{9999999999999}
\bibitem[BilBre07]{Billera} L.~Billera and F.~Brenti, \emph{Quasisymmetric functions and Kazhdan-Lusztig polynomials}, preprint 2007.
\textsf{arxiv:0710.3965}
\bibitem[BilLak01]{Billey.Lakshmibai} S.~Billey and V.~Lakshmibai, \emph{Singular loci of Schubert varieties}, Progr.~Math. {\bf
    182}(2000),
Birkh\"{a}user, Boston.
\bibitem[BilWar01]{Billey.Warrington} S.~Billey and G.~Warrington, \emph{Kazhdan-Lusztig polynomials for $321$-hexagon-avoiding
    permutations}, J.~Alg.~Comb., {\bf 13}(2) (2001), 111--136.
\bibitem[Boe88]{Boe} B.~D.~Boe, \emph{Kazhdan-Lusztig polynomials for hermitian symmetric spaces}, Trans.~Amer. Math.~Soc. {\bf
    309}(1988), 279--294.
\bibitem[BraMac01]{Braden} T.~Braden and R.~MacPherson, \emph{From moment graphs to intersection cohomology},
Math. Ann. {\bf 321}(2001), no.~3, 533--551.
  \bibitem[Bre98]{Brenti} F.~Brenti, \emph{Lattice paths and Kazhdan-Lusztig polynomials}, J.~Amer.~Math.~Soc., {\bf 11}(1998), 229--259.
\bibitem[Bre94]{Brenti} \bysame, \emph{Log-concave and unimodal sequences in algebra, combinatorics, and geometry: an update},
  Jerusalem combinatorics '93,  71--89, Contemp. Math., 178, Amer. Math. Soc., Providence, RI, 1994.
\bibitem[Bri03]{Brion} M.~Brion, \emph{Lectures on the geometry of flag varieties},
Notes de l'\'{e}cole d'\'{e}t\'{e} ``Schubert Varieties'' (Varsovie, 2003), 59 pages.
\bibitem[BruHer93]{Bruns.Herzog} W.~Bruns and J.~Herzog, \emph{Cohen-Macaulay rings}, Cambridge Studies in Advanced Mathematics, 39.
    Cambridge University Press, Cambridge, 1993. xii+403 pp.
\bibitem[Buc00]{Buch:KLR} A.~S.~Buch, \emph{A Littlewood-Richardson rule for the $K$-theory of Grassmannians},
Acta Math.~{\bf 189}(2002), no.~1, 37--78.
%\bibitem[Cor03]{Cortez} A.~Cortez, \emph{Singularit\'{e}s g\'{e}n\'{e}riques et quasi-r\'{e}solutions des vari\'{e}t\'{e}s de Schubert
%    pour
%    le groupe lin\'{e}aire}, Adv.~Math. {\bf 178}(2003), 396--445.
\bibitem[Ful92]{Fulton:Duke92} W.~Fulton, \emph{Flags, Schubert polynomials, degeneracy loci, and determinantal formulas}, Duke
    Math.~J.~{\bf 65} (1992), no.~3, 381--420.
\bibitem[IkeNar09]{Naruse} T.~Ikeda and H.~Naruse, \emph{Excited Young diagrams and equivariant Schubert calculus},
Trans.~Amer.~Math.~Soc.~{\bf 361}(2009), 5193--5221.
\bibitem[Irv88]{Irving} R.~Irving, \emph{The socle filtration of a Verma
module}, Ann.~Sci.~\'{E}cole.~Norm.~Sup.~series 4{\bf 21}(1988), no.~1, 47--65.
\bibitem[JonWoo10]{Jones.Woo} B.~Jones and A.~Woo, \emph{Kazhdan-Lusztig polynomials for cograssmannian permutations}, preprint, 2010.
\bibitem[KazLus80]{KL:proof} D.~Kazhdan and G.~Lusztig, \emph{Schubert varieties and Poincar\'{e} duality}, Proc.~Symp.~Pure.~Math., A.~M.~S., {\bf
    36}(1980), 185--203.
\bibitem[KazLus79]{Kazhdan.Lusztig} \bysame,
\emph{Representations of Coxeter Groups and Hecke Algebras},
Invent.~Math.~{\bf 53} (1979), 165--184.
\bibitem[Knu09]{Knutson:frob} A.~Knutson, \emph{Frobenius splitting, point
counting, and degeneration}, preprint, 2009. \textsf{arXiv:0911.4941}
\bibitem[KnuMil05]{KM:annals} A.~Knutson and E.~Miller, \emph{Gr\"{o}bner geometry of Schubert polynomials}, Ann.~of.~Math. (2) {\bf 161}(2005), no.~3, 1245--1318.
\bibitem[KnuMil04]{KM:subword} \bysame, \emph{Subword complexes in Coxeter groups}, Adv.~Math. {\bf 184}(2004), no.~1,
    161--176.
\bibitem[KnuMilYon09]{KMY:crelle} A.~Knutson, E.~Miller and A.~Yong, \emph{Gr\"{o}bner geometry of vertex decompositions and of flagged tableaux}, J.~Reine Angew.
    Math.~{\bf 630}(2009), 1--31.
\bibitem[KnuMilYon08]{KMY} \bysame, \emph{Tableau complexes}, Israel J.~Math., {\bf 163}(2008), 317--343.
\bibitem[KreRob05]{Kreuzer.Robbiano} M.~Kreuzer, L.~Robbiano, {Computational commutative algebra. 2}, Springer-Verlag, Berlin, 2005. x+586
    pp.
\bibitem[LakSan90]{Lakshmibai.Sandhya} V.~Lakshmibai and B.~Sandhya, \emph{Criterion
for smoothness of Schubert varieties in $SL(n)/B$}, Proc.~Indian Acad.~Sci.~Math.~Sci.~{\bf 100} (1990), no.~1, 45--52.
\bibitem[Las95]{Lascoux} A.~Lascoux, \emph{Polynomes de Kazhdan-Lusztig pour les varietes de Schubert vexillaires},
C.~R.~Acad.~Sci.~Paris Ser.~I Math. {\bf 321}(6), (1995), 667ö€Žà¤¡"1ऄ1¤770.
\bibitem[LasSch96]{treillis} A.~Lascoux and M.~-P. Sch\"{u}tzenberger, \emph{Treillis et bases des groupes de Coxeter}, Electron.~J.~Combin., {\bf 3}:2 (1996).
\bibitem[LasSch81]{LS:KL} \bysame,
\emph{Polynomes de Kazhdan $\&$ Lusztig pour les Grassmanniennes},
Ast\'{e}risque 87--88 (1981), 249--266.
% Young tableaux and
%Schur functions in algebra and geometry (Toru\'{n}, 1980).
%\bibitem[LasSch82a]{Lascoux.Schutzenberger1} \bysame,
%\emph{Polyn\^{o}mes de Schubert}, C.~R.~Acad.~Sci.~Paris S\'{e}r.~I Math.~{\bf 294} (1982),
%no.~13, 447--450.
\bibitem[LiYon10]{Li.Yong} L.~Li and A.~Yong, \emph{Some degenerations of Kazhdan-Lusztig polynomials
and multiplicities of Schubert varieties}, preprint 2010. \textsf{arXiv:1001.3437}
\bibitem[Man01a]{Manivel} L.~Manivel, \emph{Symmetric functions, Schubert polynomials and
degeneracy loci}, American Mathematical Society, Providence 2001.
\bibitem[Man01b]{Manivel:paper} \bysame, \emph{Generic singularities
of Schubert varieties}, preprint 2001. \textsf{arXiv:math.AG/0105239}.
\bibitem[MilStu04]{Miller.Sturmfels} E.~Miller and B.~Sturmfels, \emph{Combinatorial Commutative Algebra},
Graduate Texts in Mathematics Vol.~{\bf 227}, Springer-Verlag, New York, 2004.
\bibitem[Pol00]{Polo} P.~Polo, \emph{Construction of arbitrary Kazhdan-Lusztig polynomials},
Represent.~Theory  3  (1999), 90--104.
\bibitem[Ram85]{Raman} A.~Ramanathan, \emph{Schubert varieties are arithmetically Cohen-Macaulay},
Invent.~Math., {\bf 80}(1985), 283--294.
\bibitem[Rub05]{Rubey} M.~Rubey, \emph{The $h$-vector of a ladder determinantal ring
cogenerated by $2\times 2$ minors is log-concave}, J. Algebra {\bf 292}(2005), no. 2, 303--323.
\bibitem[ShiZin10]{Zinn} K.~Shigechi and
P.~Zinn-Justin, \emph{Path representation
of maximal parabolic Kazhdan-Lusztig polynomials}, preprint 2010. \textsf{arXiv.1001.1080}
%\bibitem[Woo04]{Woo:catalan} A.~Woo, \emph{Catalan numbers and Schubert polynomials for
%$w= 1(n+1)\cdots 2$}, preprint, 2004. %\textsf{arXiv:math/0407160}
\bibitem[Sta89a]{Stanley:unimodal}
R.~P.~Stanley, \emph{Log-concave and unimodal sequences in algebra, combinatorics, and geometry}.
Graph theory and its applications: East and West (Jinan, 1986), 500--535,
Ann. New York Acad. Sci., 576, New York Acad. Sci., New York, 1989.
%\bibitem[Sta89b]{Stanley:hvec} R.~P.~Stanley, \emph{Generalized $H$-vectors, intersection cohomology of toric varieties, and related
%    results}, Commutative algebra and combinatorics (Kyoto, 1985), 187--213,
%Adv. Stud. Pure Math., 11, North-Holland, Amsterdam, 1987.
\bibitem[WooYon09]{WYIII} A.~Woo and A.~Yong, \emph{A Gr\"{o}bner basis for Kazhdan-Lusztig ideals},
preprint 2009. \textsf{arxiv:0909.0564}
\bibitem[WooYon08]{WYII} \bysame, \emph{Governing singularities of Schubert varieties},
J.~Algebra, {\bf 320}(2008), no.~2, 495--520.
%\bibitem[Zel83]{Zelevinsky}A.~Zelevinsky, \emph{
%Small resolutions of singularities of Schubert varieties}, %Functional Anal. Appl. 17 (1983), no. 2, 142--144.
\end{thebibliography}
\end{document}